\documentclass{article}
\usepackage{amssymb}
\usepackage{amsmath}

\newtheorem{theorem}{Theorem}

\newtheorem{corollary}[theorem]{Corollary}

\newtheorem{definition}[theorem]{Definition}

\newtheorem{lemma}[theorem]{Lemma}

\newtheorem{proposition}[theorem]{Proposition}
\newtheorem{remark}[theorem]{Remark}

\numberwithin{equation}{section}
\newenvironment{proof}[1][Proof]{\noindent\textbf{#1.} }{\ \rule{0.5em}{0.5em}}
\input{tcilatex}
\begin{document}

\title{Isometric Embeddings via Heat Kernel}
\author{Xiaowei Wang and Ke Zhu}
\maketitle

\begin{abstract}
For any $n$-dimensional compact Riemannian manifold $\left( M,g\right) $, we
construct a canonical $t$-family of isometric embeddings $I_{t}:M\rightarrow 
\mathbb{R}^{q\left( t\right) }$, with $t>0$ sufficiently small and $q\left(
t\right) >>$ $t^{-\frac{n}{2}}$. This is done by intrinsically perturbing
the heat kernel embedding introduced in \cite{BBG}. As $t\rightarrow 0_{+}$,
asymptotic geometry of the embedded images is discussed.
\end{abstract}


\section{Introduction}

Given a $n$-dimensional Riemannian manifold $\left( M,g\right) $, one seeks
for the embeddings $u:M\rightarrow \mathbb{R}^{q}$ for some $q$ such that
the induced metric is $g$, i.e. $u^{\ast }g_{can}=g$, where $g_{can}$ is the
standard Euclidean metric in $\mathbb{R}^{q}$. This is called the \emph{%
isometric embedding} \emph{problem} and has long history, with contributions
from many people (see e.g. \cite{G2}, \cite{HH} for survey). In a celebrated
paper \cite{N2} in 1956, Nash proved the existence of global isometric
embeddings of class $C^{s}$ for $g\in C^{s}$, with $s\geq 3$ or $s=\infty $,
and dimension $q_{c}=\frac{3}{2}n\left( n+1\right) +4n$ in the compact case, 
$q=\left( n+1\right) q_{c}$ in the noncompact case.

Nash's proof used the so-called \emph{hard implicit function theorems} or
Nash-Moser technique, which involves smoothing operators in the Newton
iteration to preserve the differentiability of approximate solutions of the
isometric problem. G\"{u}nther (1989, \cite{G1}) significantly simplified
Nash's proof by inventing a \emph{new iteration scheme} for the isometric
embedding problem, such that there is no loss of differentiability in the
iteration so the usual contracting mapping theorem is enough. A good
exposition of his method is in \cite{G2}.

Nash and G\"{u}nther's isometric embedding is very flexible. It can start
from any \emph{short embedding} as the approximate solution, i.e. any
embedding $u:M\rightarrow \mathbb{R}^{q}$ such that the induced metric is
less or equal to $g$, to produce an isometric embedding. On the other hand,
such great flexibility of the initial embeddings usually makes the resulted
isometric embeddings \emph{non-canonical}. In their methods it is needed to
apply the implicit function theorem on the local \emph{coordinate charts} of
the manifolds, and in each iteration step to perturb the mapping $%
f:M\rightarrow \mathbb{R}^{q}$ to a \emph{free mapping\thinspace ,} i.e. the 
$n$ first derivative vectors $\partial _{i}f\left( x\right) $ and the $%
n\left( n+1\right) /2$ second derivative vectors $\partial _{i}\partial
_{j}f\left( x\right) $ are linearly independent at every $x$ on $M$ (see
Definition \ref{free-map}), and to estimate the right inverse of the matrix
spanned by these derivative vectors.

Motivated by K\"{a}hler geometry and conformal geometry, one would like to
find a \emph{canonical} isometric embedding of a Riemannian manifold into $%
S^{q-1}$or $\mathbb{R}^{q}$ for $q\gg 1$, such that the corresponding
geometry of the underlying Riemannian manifold is reflected via the symmetry
groups on \ $S^{q-1}$ or $\mathbb{R}^{q}$ (cf. \cite{Wa}). To achieve that,
one first has to find such canonical embeddings. In 1994, B\'{e}rard, Besson
and Gallot \cite{BBG} made progress by constructing an \textquotedblleft 
\emph{asymptotically isometric embedding}\textquotedblright\ of compact
Riemannian manifolds $M$ into the space $\ell ^{2}$ of real-valued, square
summable series, using the \emph{normalized heat kernel embedding }for $t>0:$%
\begin{equation}
\Psi _{t}:x\rightarrow \sqrt{2}\left( 4\pi \right) ^{n/4}t^{\frac{n+2}{4}%
}\cdot \left\{ e^{-\lambda _{j}t/2}\phi _{j}\left( x\right) \right\} _{j\geq
1},  \label{normalized-heat-kernel-emb}
\end{equation}%
where $\lambda _{j}$ is the $j$-th eigenvalue of the Laplacian $\Delta _{g}$
of $\left( M,g\right) $ and $\left\{ \phi _{j}\right\} _{j\geq 0}$ is the $%
L^{2}$ orthonormal eigenbasis of $\Delta _{g}$. The advantage is that the
embeddings $\Psi _{t}:M\rightarrow \mathbf{\ell }^{2}$ are \emph{canonical, }%
in the sense that they are uniquely determined by the spectral geometry of \ 
$\left( M,g\right) $. Moreover, when $t\rightarrow 0_{+}$ the embedding $%
\Psi _{t}$ tends to an isometry in the following sense: 
\begin{equation}
\Psi _{t}^{\ast }g_{can}=g+\frac{t}{3}\left( \frac{1}{2}S_{g}\cdot
g-Ric_{g}\right) +O\left( t^{2}\right) \text{,}  \label{asymp-isom}
\end{equation}%
where $g_{can}$ is the standard metric in $\ell ^{2}$, $S_{g}$ and $Ric_{g}$
are scalar and Ricci curvatures of $\left( M,g\right) $ respectively, and
the convergence is in the $C^{r}$ sense for any $r\geq 0$ (see \cite{BeGaM}
p. 213). However for any given $t>0$, $\Psi _{t}$ usually is only \emph{%
asymptotically} isometric (with an error of order $O\left( t\right) \,$).

So we are in the following situation: Nash's embedding is finite dimensional
but \emph{far from being canonical }and the heat kernel embedding is
canonical but only \emph{asymptotically isometric}. In this paper, we are
able to produce a \emph{canonical} isometric embedding into $\mathbb{R}^{q}$
for $q\gg 1$ by modifying the almost isometric embedding $\Psi _{t}$ in \cite%
{BBG} to a better approximation with error bounded by $O\left( t^{l}\right) $
for any given $l\geq 2$, and then perturbing it to an isometry by G\"{u}%
nther's implicit function theorem (\cite{G2}). Fixing two constants $\rho >0$
and $0<\alpha <1$ throughout our paper, we have our main

\begin{theorem}
\label{isom-emb} Let $\left( M,g\right) $ be a $n$-dimensional compact
Riemannian manifold with smooth metric $g$. Then

\begin{enumerate}
\item For any integer $l\geq 1$, there exist a canonical family of \emph{%
almost} isometric embeddings $\tilde{\Psi}_{t}:$ $M\rightarrow \ell ^{2}$
such that%
\begin{equation*}
\tilde{\Psi}_{t}^{\ast }g_{can}=g+O\left( t^{l}\right)
\end{equation*}%
as $t\rightarrow 0_{+}$, where the above convergence is in the $C^{r}$ sense
for any $r\geq 0$.

\item For any integer $k\geq 2$ satisfying $k+\alpha <l+\frac{1}{2}$, there
exists a constant $t_{0}>0$ depending on $k,\alpha ,l$ and $g$, such that
for $0<t\leq t_{0}$, we can truncate $\tilde{\Psi}_{t}$ to $\mathbb{R}%
^{q\left( t\right) }\subset \ell ^{2}$ and perturb it to a unique $%
C^{k,\alpha }$\ isometric embedding 
\begin{equation*}
I_{t}:M\rightarrow \mathbb{R}^{q\left( t\right) },
\end{equation*}%
where dimension $q\left( t\right) \geq t^{-\frac{n}{2}-\rho }$ or $q\left(
t\right) =\infty $, and $\left\Vert I_{t}-\tilde{\Psi}_{t}\right\Vert
_{C^{k,\alpha }\left( M\right) }=O\left( t^{l+\frac{1}{2}-\frac{k+\alpha }{2}%
}\right) $.
\end{enumerate}
\end{theorem}

The isometric embedding $I_{t}$ is\emph{\ canonical} in the sense that our $%
\tilde{\Psi}_{t}:M\rightarrow $ $\ell ^{2}$ is canonically constructed from $%
\Psi _{t}$ in \cite{BBG} (see Section \ref{heat-kernel-l2}), and our
implicit function theorem only uses the \emph{intrinsic} information of $%
\left( M,g\right) $. More precisely, the smoothing operator needed in G\"{u}%
nther's iteration scheme was constructed \emph{directly} from $g$. The
iteration attempts to adjust $\tilde{\Psi}_{t}$ to the nearest isometric
embedding in each step with a \emph{unique} minimal movement.

Our method has the following advantages. The heat kernel embedding $\Psi
_{t}:M\rightarrow \mathbb{R}^{q\left( t\right) }$ is automatically a \emph{%
free mapping} for small $t$. Furthermore, the row vectors $\left\{ \partial
_{i}\Psi _{t}\left( x\right) \right\} _{i=1}^{n}$ and $\left\{ \partial
_{i}\partial _{j}\Psi _{t}\left( x\right) \right\} _{1\leq i\leq j\leq n}$
span a matrix $P\left( \Psi _{t}\right) $ with an \emph{explicit right
inverse bound} (Corollary \ref{C2a-operator-norm}) on the \emph{whole} $M$.
(These nice properties are inherited to $\tilde{\Psi}_{t}$ as well). There
is no need of sophisticated perturbation arguments in local charts to
achieve the right inverse bound as in Nash and G\"{u}nther's methods.

We have good control of the second fundamental form and mean curvature of
the embedded images $\Psi _{t}\left( M\right) $ and $I_{t}\left( M\right) $
in $\mathbb{R}^{q\left( t\right) }$. For simplicity we only state the mean
curvature part:

\begin{proposition}
\label{second-funda-form} Let $M$ be a compact Riemannian manifold with
smooth metric. For any $x$ on $M$, let $H\left( x,t\right) $ be the mean
curvature vector at $\Psi _{t}\left( x\right) $ (or $I_{t}\left( x\right) $)
in $\mathbb{R}^{q\left( t\right) }$. Then as $t\rightarrow 0_{+}$, 
\begin{equation*}
\sqrt{t}\left\vert H\left( x,t\right) \right\vert \rightarrow \sqrt{\frac{n+2%
}{2n}}.
\end{equation*}
\end{proposition}

The second fundamental form also has certain normal form as $t\rightarrow
0_{+}$. (See Corollary \ref{second-fund-form-mean-curvature} and Remark \ref%
{Comments-2nd-funda-form}).

\bigskip Finally,\ we make a few remarks about our approach. First, in \cite%
{BBG} \ the authors also constructed the heat kernel embedding into the
infinite dimensional\emph{\ unit sphere} $S^{\infty }\subset \ell ^{2}$, by%
\begin{equation*}
K_{t}:x\rightarrow \left. \left\{ e^{-\lambda _{j}t/2}\phi _{j}\left(
x\right) \right\} _{j\geq 1}\right/ \left( \Sigma _{j\geq 1}e^{-\lambda
_{j}t}\phi _{j}^{2}\left( x\right) \right) ^{1/2},
\end{equation*}%
with the asymptotic behavior $K_{t}^{\ast }g_{can}=\frac{1}{2t}\left( g-%
\frac{t}{3}Ric_{g}+O\left( t^{2}\right) \right) $ as $t\rightarrow 0_{+}$,
where $g_{can}$ is the standard metric in $\ell ^{2}$. There is a parallel
version of Theorem \ref{isom-emb} for $S^{q\left( t\right) }$ provable by
our method, i.e. we can truncate and perturb $K_{t}$ to $\Upsilon
_{t}:\left( M,g\right) \rightarrow S^{q\left( t\right) }\subset \mathbb{R}%
^{q\left( t\right) +1}$ such that $\Upsilon _{t}^{\ast }g_{can}=\frac{1}{2t}g
$. But for the sake of simplicity, we only write the one for $\mathbb{R}%
^{q\left( t\right) }$. In the sequel to this paper, we will extend our
method to construct a canonical\emph{\ conformal embedding} $\Theta
_{t}:\left( M,g\right) \rightarrow S^{q}\subset \mathbb{R}^{q+1}$ that keeps
more information of $K_{t}$. It will be interesting to see the relation to
the conformal volume defined in \cite{LY2}. Second, we want to point out
that it is not our main emphasis to optimize the dimension $q\left( t\right)
\geq t^{-\frac{n}{2}-\rho }$, since the lower the embedding dimension is,
the less canonical our map is. Third, it will not make too much difference
if one uses the heat kernel for a perturbation of the Laplacian operator
(cf. \cite{Gil}), but we choose to use the canonical one. Last, we only deal
with compact Riemannian manifolds, while it is likely that our method can be
extended to Riemannian manifolds with boundary or even complete Riemannian
manifolds (with suitable condition at infinity). We leave these to future
investigation. We notice there are several related works on embeddings of
compact Riemannian manifolds by eigenfunctions (with various weights) and
heat kernels recently, e.g. \cite{Ni}, \cite{Wu}, \cite{P} and \cite{Po}. 

The organization of the paper is follows: In Section \ref{heat-kernel-l2} we
review the heat kernel embedding $\Psi _{t}:M\rightarrow \ell ^{2}$ in \cite%
{BBG}. Then we modify $\Psi _{t}$ to get improved error to isometry, and
truncate the embedding to $\mathbb{R}^{q}\subset \ell ^{2}$ and estimate the
remainder. In Section \ref{Gunther-IFT} we recall the matrix $E\left(
u\right) $ appeared in the linearization of the isometric embedding problem,
and review G\"{u}nther's iteration scheme and the implicit function theorem.
In Section \ref{Sec:uniform-indep} we give higher derivative estimates of $%
\Psi _{t}$ using the off-diagonal heat kernel expansion method, and
establish the crucial uniform linear independence property of the matrix $%
E\left( \Psi _{t}\right) $. Then we give the operator norm estimate of $%
E\left( \Psi _{t}\right) $. In Section \ref{Qudratic-estimate} we establish
the uniform quadratic estimate of the nonlinear operator $Q\left( u\right) $
in the isometric embedding problem for all $\mathbb{R}^{q}$. In Section \ref%
{IFT} we apply G\"{u}nther's implicit function theorem to the modified $\Psi
_{t}$ to obtain isometric embeddings $I_{t}:\left( M,g\right) \rightarrow 
\mathbb{R}^{q\left( t\right) }$. The geometry of $I_{t}\left( M\right) $ is
close to that of $\Psi _{t}\left( M\right) $ by our error estimate. In
Section \ref{asymp-geom}, we derive the asymptotic formulae of the second
fundamental form and mean curvature of the embedded images $\Psi _{t}\left(
M\right) $ as $t\rightarrow 0_{+}$. In Section \ref{examples} we illustrate
our method by explicit calculations on $M=S^{1}$. In the Appendix we make
the constant in G\"{u}nther's implicit function theorem explicit, and
discuss the minimal embedding dimension of our method.

\textbf{Convention}: In this paper, unless otherwise remarked, the constant $%
C$ only depends on $\left( M,g\right) $, its dimension $n$, and $k,\alpha $
in the $C^{k,\alpha }$-H\"{o}lder norm, but not on $t$, and $q$ of $\mathbb{R%
}^{q}$. In a sequence of inequalities, the constant $C$ in successive
appearances can be assumed to increase. The two constants $\rho >0$ and $%
0<\alpha <1$ are fixed throughout the paper. The constant $k$ in the $%
C^{k,\alpha }$-norm, and the constant $l$ in the error term $O\left(
t^{l}\right) $ should not be confused with the indices $k,l$ $\left( 1\leq
k,l\leq n\right) $ in partial derivatives like $\partial _{k}$ and $\partial
_{k}\partial _{l}$.

\textbf{Acknowledgement.} We thank Gang Liu for an inspiring discussion on
the off-diagonal heat kernel expansion, Deane Yang for helpful email
correspondence, and Jiaping Wang for helpful discussions. Both authors would
like to thank S.-T. Yau for his influence on canonical embeddings. The
second named author would like to thank Clifford Taubes for his interest and
support. He also thanks Peng Gao, Baosen Wu and Qingchun Ji for discussions
on the embedding problem in wider contexts. The work of K. Zhu is partially
supported by the grant of Clifford Taubes from the National Science
Foundation.

\section{The heat kernel embedding into $\ell ^{2}$ and modifications$\label%
{heat-kernel-l2}$}

Let $\ell ^{2}$ be the Hilbert space of real series $\left\{ a_{i}\right\}
_{i\geq 1}$ such that $\sum_{i=1}^{\infty }a_{i}^{2}<\infty $, and $g_{can}$
be the standard metric in $\ell ^{2}$. Let $\left( M,g\right) $ be a $n$%
-dimensional compact Riemannian manifold with smooth metric $g$, and $%
\left\{ \phi _{j}\left( x\right) \right\} _{j\geq 0}\subset C^{\infty
}\left( M\right) $ be a $L^{2}$-orthonormal basis of real eigenfunctions of
the Laplacian of $M$, i.e. for eigenvalues $0=\lambda _{0}<\lambda _{1}\leq
\lambda _{2}\leq \cdots $, $\Delta _{g}\phi _{j}=\lambda _{j}\phi _{j}$, and 
$\int_{M}\phi _{i}\phi _{j}dvol_{_{g}}=\delta _{ij}$ for $\forall i,j$. The
heat kernel of $\left( M,g\right) $ is 
\begin{equation*}
H\left( t,x,y\right) =\Sigma _{s=1}^{\infty }e^{-\lambda _{s}t}\phi
_{s}\left( x\right) \phi _{s}\left( y\right) 
\end{equation*}
for $x,y\in M$ and $t>0$.

\begin{definition}
We call the family of maps 
\begin{equation*}
\Phi _{t}:%
\begin{array}{ccc}
M & \longrightarrow  & \ell ^{2} \\ 
x & \longmapsto  & \left\{ e^{-\lambda _{j}t/2}\phi _{j}\left( x\right)
\right\} _{j\geq 1}%
\end{array}%
\text{ for }t>0
\end{equation*}%
the \emph{heat kernel embeddings}, and call $\Psi _{t}=\sqrt{2}\left( 4\pi
\right) ^{n/4}t^{\frac{n+2}{4}}\cdot \Phi _{t}$ the \emph{normalized heat
kernel embeddings}.
\end{definition}

From the definition we clearly have $H\left( t,x,y\right) =\left\langle \Phi
_{t}\left( x\right) ,\Phi _{t}\left( y\right) \right\rangle $, where $%
\left\langle ,\right\rangle $ is the standard inner product in $\ell ^{2}.$

In \cite{BBG}, B\'{e}rard, Besson and Gallots introduced the above maps and
proved the following

\begin{theorem}
(\cite{BBG} Theorem 5) As $t\rightarrow 0_{+}$, there is an expansion%
\begin{equation}
\Psi _{t}^{\ast }g_{can}=g+\sum_{i=1}^{l}t^{i}A_{i}\left( g\right) +O\left(
t^{l+1}\right) ,  \label{BBG-expansion}
\end{equation}%
with 
\begin{equation*}
A_{1}=\frac{1}{3}\left( \frac{1}{2}S_{g}\cdot g-Ric_{g}\right) ,
\end{equation*}%
where $A_{i}$'s are universal polynomials of the covariant differentiations
of the metric $g$ and its curvature tensors up to order $2i$, and the
convergence in $\left( \ref{BBG-expansion}\right) $ is in the $C^{r}$ sense
for any $r\geq 0$.
\end{theorem}

As a direct consequence, we have the following singular perturbation result
parallel to Theorem 26 in \cite{Don}. Denote the space of symmetric $2$%
-tensors on $M$ by $\Gamma \left( \mathrm{Sym}^{\otimes 2}\left( T^{\ast
}M\right) \right) $.

\begin{proposition}
\label{Higher-order-error}For any $l\geq 1$, there are $h_{i}\in \Gamma
\left( \mathrm{Sym}^{\otimes 2}\left( T^{\ast }M\right) \right) $ ($1\leq
i\leq l-1$) such that for the following family of metrics 
\begin{equation*}
g\left( s\right) =g+\sum_{i=1}^{l-1}s^{i}h_{i}\text{,}
\end{equation*}%
the induced metric from the heat kernel embeddings $\Psi _{t,g\left(
s\right) }^{\ast }:\left( M,g\left( s\right) \right) \rightarrow \ell ^{2}$
satisfies the following estimate 
\begin{equation}
\left\Vert \Psi _{t,g\left( t\right) }^{\ast }g_{can}-g\right\Vert
_{C^{r}\left( g\right) }\leq C\left( g,l,r\right) t^{l},
\label{enhanced-error}
\end{equation}%
for any $r\geq 0$, where the constant $C\left( g,l,r\right) $ depends only
on $l,r$ and the geometry of $\left( M,g\right) $.
\end{proposition}

\begin{proof}
Let us assume that 
\begin{equation}
g\left( s\right) =g+\sum_{i=1}^{l-1}s^{i}h_{i}\text{ with }h_{i}\in \Gamma
\left( \mathrm{Sym}^{\otimes 2}\left( T^{\ast }M\right) \right) ,
\label{gs-curve}
\end{equation}%
where $h_{i}$'s are to be determined. Then by $\left( \ref{BBG-expansion}%
\right) $, for metric $g\left( s\right) $ we have 
\begin{equation}
G\left( s,t\right) :=\Psi _{t,g\left( s\right) }^{\ast }g_{can}=g\left(
s\right) +tA_{1}\left( g\left( s\right) \right) +t^{2}A_{2}\left( g\left(
s\right) \right) +\cdots   \label{Gst}
\end{equation}%
with $A_{i}$'s are universal polynomials of the\ covariant differentiations
of any metric and its curvature tensors up to order $2i.$ Using the Taylor
expansion of $A_{i}\left( g\left( s\right) \right) $ at $s=0$, we have 
\begin{equation*}
A_{i}\left( g\left( s\right) \right) =A_{i}\left( g\right)
+\sum_{j=1}^{\infty }A_{i,j}\left( h_{1},\cdots h_{j}\right) s^{j},\text{ }
\end{equation*}%
where each%
\begin{equation*}
A_{i,j}\left( h_{1},\cdots h_{j}\right) :=\left. \frac{\partial ^{j}}{%
\partial s^{j}}\right\vert _{s=0}\frac{1}{j!}A_{i}\left( g\left( s\right)
\right) 
\end{equation*}%
is a universal polynomial of the\ covariant differentiations of the metric $g
$ and its curvature tensors, and is multi-linear in $h_{1},\cdots h_{j}$ by
the chain rule. Putting this into $\left( \ref{BBG-expansion}\right) $ we
have in the $C^{r}$ norm convergence%
\begin{eqnarray*}
G\left( t,t\right)  &=&\left( g+th_{1}+t^{2}h_{2}+\cdots \right)  \\
&&+t\left( A_{1}\left( g\right) +A_{1,1}\left( h_{1}\right) t+A_{1,2}\left(
h_{1},h_{2}\right) t^{2}+\cdots \right)  \\
&&+t^{2}\left( A_{2}\left( g\right) +A_{2,1}\left( h_{1}\right)
t+A_{2,2}\left( h_{1},h_{2}\right) t^{2}+\cdots \right) +\cdots 
\end{eqnarray*}
\begin{equation}
+t^{l}\left( A_{l-1}\left( g\right) +A_{l-1,1}\left( h_{1}\right)
t+A_{l-1,2}\left( h_{1},h_{2}\right) t^{2}+\cdots \right) +O\left(
t^{l}\right) .  \label{G-metric-expansion}
\end{equation}%
Now we let $h_{1}=-A_{1}\left( g\right) $, $h_{2}=-A_{1,1}\left(
h_{1}\right) -A_{2}\left( g\right) $, and in general let 
\begin{equation*}
h_{j}:=-A_{1,j-1}\left( h_{1},\cdots ,h_{j-1}\right) -A_{2,j-2}\left(
h_{1},\cdots ,h_{j-2}\right) -\cdots A_{j-1,1}\left( h_{1}\right)
-A_{j}\left( g\right) 
\end{equation*}%
inductively for $1\leq j\leq l-1$. Then we are able to construct a curve $%
g\left( s\right) \subset \Gamma \left( \mathrm{Sym}^{\otimes 2}\left(
T^{\ast }M\right) \right) $ by $\left( \ref{gs-curve}\right) $ such that 
\begin{equation*}
\Psi _{t,g\left( t\right) }^{\ast }g_{can}=G\left( t,t\right) =g+O\left(
t^{l}\right) 
\end{equation*}%
in the $C^{r}$ sense for any $r\geq 0$, as we claimed.
\end{proof}

\begin{definition}
\label{modified-heat-emb} (Modified heat kernel embedding) We call the $\Psi
_{t,g\left( t\right) }:M\rightarrow \ell ^{2}$ constructed above as the 
\emph{modified heat kernel embedding}, and denote%
\begin{equation}
\tilde{\Psi}_{t}:=\Psi _{t,g\left( t\right) }.  \label{modified}
\end{equation}
\end{definition}

To get the embedding into $\mathbb{R}^{q}$, let 
\begin{equation}
\Pi _{q}:\ell ^{2}\longrightarrow \mathbb{R}^{q}  \label{Piq}
\end{equation}%
be the projection of $\ell ^{2}$ to the first $q$ components. To get a
finite dimensional isometric embedding, we introduce the truncated embedding 
\begin{equation*}
\Psi _{t}^{q\left( t\right) }:=\Pi _{q}\circ \Psi _{t}:\left( M,g\right)
\longrightarrow \ell ^{2}\overset{\Pi _{q}}{\longrightarrow }\mathbb{R}%
^{q\left( t\right) }.
\end{equation*}

\begin{remark}
\bigskip \label{AnalyticalFamily} Since the metrics $g\left( s\right) $
constructed in $\left( \ref{gs-curve}\right) $ depend on $s$ analytically,
given any $\mu _{0}>0$ not in the spectrum of $\Delta _{g}$, there exists $%
\delta _{0}>0$, such that for $\Delta _{g_{s}}$ with $0\leq s<\delta _{0}$,
for their eigenvalues  $0=\lambda _{0}\leq \lambda _{1}\left( s\right) \leq
\cdots \leq \lambda _{j_{0}}\left( s\right) <\mu _{0}$, the total
multiplicity $j_{0}$ is independent on $s$,  and each $\lambda _{j}\left(
s\right) $ ($0\leq j\leq j_{0}$) depends on $s$ analytically. Furthermore,
we can choose the eigenfunctions $\phi _{j}\left( s,x\right) $ of $\Delta
_{g_{s}}$ associated with these $\lambda _{j}\left( s\right) $ such that
they are orthonormal in $L^{2}\left( M,g_{s}\right) $, and depend on $s$
analytically (see \cite{A}, Lemma 2.1, and earlier \cite{R}). Therefore, for 
$0\leq t<\delta _{0}$, the truncated heat kernel mapping $\Psi
_{t}^{j_{0}}:M\rightarrow \mathbb{R}^{j_{0}}$ can be made depending on $t$
analytically. 
\end{remark}

In order to estimate the truncated tail, we recall the following well-known
derivative estimates of eigenfunctions $\phi _{j}$, and extend it\ to the H%
\"{o}lder derivative setting. \bigskip 

\begin{lemma}
\label{eigen-deri-est} For any integer $k\geq 0$ and $0<\alpha <1$, we have%
\begin{eqnarray}
\left\Vert \nabla ^{\left( k\right) }\phi _{j}\right\Vert _{C^{0}\left(
M\right) } &\leq &C\left( k,g\right) \lambda _{j}^{\frac{n+2k}{4}},
\label{eigen-Schauder2} \\
\left\Vert \phi _{j}\right\Vert _{C^{k,\alpha }\left( M\right) } &\leq
&C\left( k,\alpha ,g\right) \lambda _{j}^{\frac{n+2k+2\alpha }{4}}.
\label{eigen-Schauder1}
\end{eqnarray}%
for some positive constants $C\left( k,g\right) $ and $C\left( k,\alpha
,g\right) $.
\end{lemma}

\begin{proof}
The estimate $\left( \ref{eigen-Schauder2}\right) $ is in Theorem 17.5.3 of 
\cite{H2} and Theorem 1 of \cite{X} (when $k=0$ in earlier \cite{H1} and 
\cite{S}, and $k=1$ in \cite{LY1}). We only prove $\left( \ref%
{eigen-Schauder1}\right) $. Since $\lambda _{j}\rightarrow +\infty $ as $%
j\rightarrow \infty $, starting from some $j$ we must have $\lambda _{j}^{-%
\frac{1}{2}}$ less than the injective radius of $\left( M,g\right) $.
Without loss of generality we assume this holds from $j=1$. We consider two
cases:

For $x,y\in M$ with $d\left( x,y\right) $\thinspace $\leq \lambda _{j}^{-%
\frac{1}{2}}$, we have%
\begin{eqnarray*}
\left\vert \frac{\nabla ^{\left( k\right) }\phi _{j}\left( x\right) -\nabla
^{\left( k\right) }\phi _{j}\left( y\right) }{\left( \text{$d$}\left(
x,y\right) \right) ^{\alpha }}\right\vert  &=&\left\vert \frac{\nabla
^{\left( k\right) }\phi _{j}\left( x\right) -\nabla ^{\left( k\right) }\phi
_{j}\left( y\right) }{\text{$d$}\left( x,y\right) }\right\vert \left( \text{$%
d$}\left( x,y\right) \right) ^{1-\alpha } \\
&\leq &C\left( g\right) \left\vert \nabla ^{\left( k+1\right) }\phi
_{j}\right\vert _{C^{0}\left( M\right) }\lambda _{j}^{-\frac{1-\alpha }{2}}
\\
\text{(by }\left( \ref{eigen-Schauder2}\right) \text{)} &\leq &C\left(
g\right) C\left( k+1,g\right) \lambda _{j}^{\frac{n+2k+2\alpha }{4}}\text{,}
\end{eqnarray*}%
where the constant $C\left( g\right) $ only depends on $g$.

For $x,y\in M$ with $d\left( x,y\right) $\thinspace $\geq \lambda _{j}^{-%
\frac{1}{2}}$, we have 
\begin{equation*}
\left\vert \frac{\nabla ^{\left( k\right) }\phi _{j}\left( x\right) -\nabla
^{\left( k\right) }\phi _{j}\left( y\right) }{\left( \text{$d$}\left(
x,y\right) \right) ^{\alpha }}\right\vert \leq \frac{2\left\vert \nabla
^{\left( k\right) }\phi _{j}\right\vert _{C^{0}\left( M\right) }}{\lambda
_{j}^{-\frac{\alpha }{2}}}\text{(by }\left( \ref{eigen-Schauder2}\right) 
\text{)}\leq 2C\left( k,g\right) \lambda _{j}^{\frac{n+2k+2\alpha }{4}}.
\end{equation*}

Combining the two cases and letting $C\left( k,\alpha ,g\right) :=\max
\left\{ C\left( g\right) C\left( k+1,g\right) ,2C\left( k,g\right) \right\} $%
, we obtain $\left( \ref{eigen-Schauder1}\right) $.
\end{proof}

\begin{proposition}
\label{isom-truncation-high-jet} Let $\left\{ g_{s}\right\} _{s\in K}$ be a
compact family of smooth metrics on a compact $n$-dimensional Riemannian
manifold $M$, where $g_{s}$ depends on $s$ smoothly. Given $x\in M$, let $%
\left\{ x^{k}\right\} _{1\leq k\leq n}$ be the normal coordinates in its
neighborhood. Then for any multiple-indices $\overrightarrow{\alpha }$ and $%
\overrightarrow{\beta }$, and $q\left( t\right) \geq t^{-\left( \frac{n}{2}%
+\rho \right) }$, 
\begin{equation}
\Sigma _{j\geq q\left( t\right) +1}e^{-\lambda _{j}t}D^{\overrightarrow{%
\alpha }}\phi _{j}\left( x\right) D^{\overrightarrow{\beta }}\phi _{j}\left(
x\right) \leq C\exp \left( t^{-\frac{\rho }{n}}\right) 
\label{high-deri-truncate}
\end{equation}%
for any $l\geq 1$. The convergence is uniform for $x\in M$ and $s\in K$ in
the $C^{r}$-norm for any $r\geq 0$.
\end{proposition}

\begin{proof}
We first prove the proposition when $K$ consists of a \emph{single} metric $%
g $. By Lemma \ref{eigen-deri-est} we have for for any multi-index $%
\overrightarrow{\alpha }$, 
\begin{equation}
\left\Vert D^{\overrightarrow{\alpha }}\phi _{j}\right\Vert _{C^{0}\left(
M\right) }\leq C\left( \overrightarrow{\alpha },g\right) \lambda _{j}^{\frac{%
n+2\left\vert \overrightarrow{\alpha }\right\vert }{4}}
\label{high-eigen-derivative}
\end{equation}%
for some constant $C\left( \overrightarrow{\alpha },g\right) $, with $%
\left\vert \overrightarrow{\alpha }\right\vert $ being the degree of $%
\overrightarrow{\alpha }$. Combining the \emph{Weyl's asymptotic formula}
(p.9 in \cite{Ch}) for eigenvalues on compact manifolds $\left( M,g\right) $
that 
\begin{equation}
\lambda _{j}\thicksim \frac{4\pi ^{2}}{\left( \omega _{n}Vol(M)\right) ^{%
\frac{2}{n}}}j^{\frac{2}{n}}\geq A\left( g\right) j^{\frac{2}{n}}
\label{Weyl}
\end{equation}%
for some constant $A:=A\left( g\right) $ as $j\rightarrow \infty $ (where $%
\omega _{n}$ is the volume of the unit ball in $\mathbb{R}^{n}$), we have%
\begin{eqnarray*}
&&\left\vert \Sigma _{j\geq q\left( t\right) }e^{-\lambda _{j}t}D^{%
\overrightarrow{\alpha }}\phi _{j}\left( x\right) D^{\overrightarrow{\beta }%
}\phi _{j}\left( x\right) \right\vert \\
&\leq &C\Sigma _{j\geq q\left( t\right) }\left( j^{\frac{2}{n}}\right) ^{%
\frac{n+\left\vert \overrightarrow{\alpha }\right\vert +\left\vert 
\overrightarrow{\beta }\right\vert }{2}}e^{-Aj^{\frac{2}{n}}t}\leq
C\int_{q\left( t\right) }^{\infty }j^{\frac{n+\left\vert \overrightarrow{%
\alpha }\right\vert +\left\vert \overrightarrow{\beta }\right\vert }{n}%
}e^{-Aj^{\frac{2}{n}}t}dj \\
&\leq &Ct^{-\left( \frac{2n+\left\vert \overrightarrow{\alpha }\right\vert
+\left\vert \overrightarrow{\beta }\right\vert }{2}\right) }\int_{A\left(
q\left( t\right) \right) ^{\frac{2}{n}}t}^{\infty }\mu ^{\frac{2n+\left\vert 
\overrightarrow{\alpha }\right\vert +\left\vert \overrightarrow{\beta }%
\right\vert -2}{2}}e^{-\mu }d\mu \text{ (}\mu =Aj^{\frac{2}{n}}t\text{)} \\
&\leq &C\exp \left( t^{-\frac{\rho }{n}}\right) ,
\end{eqnarray*}%
where we have used $q\left( t\right) \geq t^{-\left( \frac{n}{2}+\rho
\right) }$, and $\mu ^{\sigma }=o\left( e^{\frac{\mu }{2}}\right) $ as $\mu
\rightarrow \infty $ for any fixed $\sigma >0$. Therefore we have proved $%
\left( \ref{high-deri-truncate}\right) $ in the $C^{0}$-convergence. The $%
C^{r}$-convergence of $\left( \ref{high-deri-truncate}\right) $ follows by
the Leibniz rule, adding the indices $\overrightarrow{\alpha }$ and $%
\overrightarrow{\beta }$ by $\overrightarrow{\gamma }$ \ with $\left\vert 
\overrightarrow{\gamma }\right\vert \leq r$ in the above argument.

For a compact family of metrics $\left\{ g_{s}\right\} _{s\in K}$ smoothly
depending on $s$, notice the constant $C\left( \overrightarrow{\alpha }%
,g_{s}\right) $ in $\left( \ref{high-eigen-derivative}\right) $ has a
uniform upper bound, and the constant $A\left( g_{s}\right) $ in $\left( \ref%
{Weyl}\right) $ has a uniform positive lower bound for all $s\in K$, because
they are determined by the following geometric quantities continuously
depending on $s$: the dimension, the curvature bound, the diameter and
volume of $\left( M,g_{s}\right) $ (see Remark \ref{effective-truncation}).
So the truncation estimate $\left( \ref{high-deri-truncate}\right) $ can be
made uniform for all $s\in K$.
\end{proof}

\begin{corollary}
\label{isom-truncation}Given any $l\geq 1$, for $q=q\left( t\right) \geq
Ct^{-\left( \frac{n}{2}+\rho \right) }$, the truncated modified heat kernel
embedding $\tilde{\Psi}_{t}^{q\left( t\right) }:\left( M,g\right)
\rightarrow \mathbb{R}^{q\left( t\right) }$ still satisfies the asymptotic
formula%
\begin{equation*}
\left( \tilde{\Psi}_{t}^{q\left( t\right) }\right) ^{\ast }g_{can}=g+O\left(
t^{l}\right) 
\end{equation*}%
in the $C^{r}$-sense for any $r\geq 0$.
\end{corollary}

\begin{proof}
From Proposition \ref{Higher-order-error} we have in the $C^{r}$-sense 
\begin{equation*}
\left( \tilde{\Psi}_{t}\right) ^{\ast }g_{can}=g+O\left( t^{l}\right) .
\end{equation*}%
To see it still holds after truncating $\tilde{\Psi}_{t}$ to $\tilde{\Psi}%
_{t}^{q\left( t\right) }$, let $D^{\overrightarrow{\alpha }}=\nabla _{i}$, $%
D^{\overrightarrow{\beta }}=\nabla _{j}$ for local normal coordinates, and $%
\left\{ g_{s}\right\} _{s\in \left[ 0,t_{0}\right] }$ be the compact family
of metrics defined in $\left( \ref{gs-curve}\right) $, and then apply the
above Proposition.
\end{proof}

\begin{remark}
\label{effective-truncation}\bigskip To get an effective truncation of $\ell
^{2}$ to $\mathbb{R}^{q}$, it is useful to have an estimate of the $j$-th
eigenvalue $\lambda _{j}$ of $M$ in terms of geometric quantities of $\left(
M,g\right) $, since the Weyl asymptotic formula $\left( \ref{Weyl}\right) $
does not tell how fast the $\lambda _{j}$ converges to its limit in $\left( %
\ref{Weyl}\right) $. Given a real number $\Lambda $, for all $n$-dimensional
compact Riemannian manifolds $\left( M,g\right) $ satisfying Ricci curvature 
$Ric_{g}\geq \Lambda g$ and diameter bounded by $D$, there exists a constant 
$A\left( n,\Lambda ,D\right) >0$ such that $\lambda _{j}\geq A\left(
n,\Lambda ,D\right) j^{\frac{2}{n}}$ (\cite{SY}, \cite{Gr2}, \cite{BBG}).
(Similar lower bound of $\lambda _{j}$ was established in earlier \cite{LY1}
under stronger assumptions). The upper bound $\lambda _{j}\leq B\left(
n,\Lambda ,D\right) j^{\frac{2}{n}}$ was established in \cite{LY1}.

The estimate of $\left\Vert \phi _{j}\right\Vert _{C^{k,\alpha }\left(
M\right) }$ for $k\geq 2$ can be reduced to $\left\Vert \phi _{j}\right\Vert
_{C^{1}\left( M\right) }$ by inductively using the elliptic estimate%
\begin{equation*}
\left\Vert \phi _{j}\right\Vert _{C^{k,\alpha }\left( M\right) }\leq
C_{e}\left( \left\Vert \Delta _{g}\phi _{j}\right\Vert _{C^{k-1,\alpha
}\left( M\right) }+\left\Vert \phi _{j}\right\Vert _{C^{0}\left( M\right)
}\right) ,
\end{equation*}%
where the constant $C_{e}$ depends on $n,D,Vol\left( M\right) $ and the
sectional curvature bound $K$. In $\left\Vert \phi _{j}\right\Vert
_{C^{1}\left( M\right) }\leq C\left( 1,g\right) \lambda _{j}^{\frac{n+1}{2}}$%
, the constant $C\left( 1,g\right) $ depends on the these quantities too
(e.g. \cite{WaZh}). Hence $n,K,D$ and $Vol\left( M\right) $ determine $%
C\left( k,\alpha ,g\right) $. 
\end{remark}

\begin{remark}
Recently \cite{Po} studied a similar almost isometric embedding of compact
Riemannian manifolds into Euclidean spaces via heat kernel plus certain
recording points on $M$, with weaker regularity assumption on $g$. The
embedding dimension is controlled by similar geometric quantities in Remark %
\ref{effective-truncation}.  
\end{remark}

\section{G\"{u}nther's iteration for isometric embedding\label{Gunther-IFT}}

\subsection{\protect\bigskip The perturbation problem and free mappings}

To solve the isometric embedding problem $du\cdot du=g$, Nash studied the
perturbation problem $d\left( u+v\right) \cdot d\left( u+v\right) =du\cdot
du+f$ for small symmetric $2$-tensors $f$. In local coordinates $\left\{
x_{i}\right\} _{1=1}^{n}$, the perturbation $v:M\rightarrow \mathbb{R}^{q}$
should satisfy $\partial _{i}u\cdot \partial _{j}v+\partial _{j}u\cdot
\partial _{i}v+\partial _{i}v\cdot \partial _{j}v=f_{ij}$. Imposing the
condition $\partial _{i}u\cdot v=0$, the equation becomes the system%
\begin{equation}
\partial _{i}u\cdot v=0\text{, \ \ \ }\partial _{j}\partial _{i}u\cdot v=-%
\frac{1}{2}f_{ij}-\frac{1}{2}\partial _{i}v\cdot \partial _{j}v.
\label{Perturbation-device}
\end{equation}%
The \emph{linear part }of the system is determined by a matrix whose row
vectors are the $\frac{n\left( n+3\right) }{2}$ vectors $\partial
_{i}u\left( x\right) $ and $\partial _{j}\partial _{k}u\left( x\right) $ in $%
\mathbb{R}^{q}$. This motivates the

\begin{definition}
\label{free-map}(Free mapping) A $C^{2}$ map $u:M\rightarrow \mathbb{R}^{q}$
(including $\ell ^{2}$, if $q=\infty $) is called a \emph{free mapping} if
the $\frac{n\left( n+3\right) }{2}$ vectors $\left\{ \partial _{i}u\left(
x\right) ,\partial _{j}\partial _{k}u\left( x\right) \right\} _{1\leq
i,j,k\leq n}$ in $\mathbb{R}^{q}$ are linearly independent at any $x\in M$,
where $\partial _{i}$ is the derivative with respect to a coordinate $%
\left\{ x_{i}\right\} _{i=1}^{n}$ of $M$ near $x$. (Note this property is 
\emph{independent} on choice of coordinates).
\end{definition}

\subsection{$C^{k,\protect\alpha }$ norms for $\mathbb{R}^{q}$-valued
functions\label{Schauder-norm-Rq}}

We first define the $C^{k,\alpha }$ norms for $\mathbb{R}^{q}$\emph{-valued
functions} for any integer $k\geq 0$ and $\alpha \in \left( 0,1\right) $.
Since our $q=q\left( t\right) \geq Ct^{-\frac{n}{2}-\rho }\rightarrow \infty 
$ as $t\rightarrow 0_{+}$, several equivalent $C^{k,\alpha }$ norms for any
fixed $q$ will diverge from each other as $q\rightarrow \infty $. To get the
uniform quadratic estimate for \emph{all} $q$, we will carefully choose the
definition of the $C^{k,\alpha }$ norm.

\begin{definition}
\label{Schauder-norm-vec-valued-fcn}Let $f:M\rightarrow \mathbb{R}^{q}$ (or $%
\ell ^{2}$, if $q=\infty $) be a $\mathbb{R}^{q}$-valued function $f=\left(
f_{1},\cdots ,f_{q}\right) $, where each $f_{j}:M\rightarrow \mathbb{R}$. We
let $\left\vert \cdot \right\vert $ be the standard Euclidean norm in $%
\mathbb{R}^{q}$, $\nabla $ be the covariant derivative of $\left( M,g\right) 
$, $\beta \geq 0$ be an integer, and let%
\begin{eqnarray}
\left\Vert \nabla ^{\beta }f\right\Vert _{C^{0}\left( M,\mathbb{R}%
^{q}\right) } &=&\sup_{x\in M}\left( \Sigma _{j=1}^{q}\left\vert \nabla
^{\beta }f_{j}\left( x\right) \right\vert ^{2}\right) ^{1/2},  \notag \\
\left\Vert f\right\Vert _{C^{k}\left( M,\mathbb{R}^{q}\right) } &=&\Sigma
_{0\leq \beta \leq k}\left\Vert \nabla ^{\beta }f\right\Vert _{C^{0}\left( M,%
\mathbb{R}^{q}\right) },  \notag \\
\left[ f\right] _{\alpha ,M;\mathbb{R}^{q}} &=&\sup_{x\neq y\in M}\frac{%
\left\vert f\left( x\right) -f\left( y\right) \right\vert }{\text{dist}%
\left( x,y\right) ^{\alpha }},  \notag \\
\left\Vert f\right\Vert _{C^{k,\alpha }\left( M,\mathbb{R}^{q}\right) }
&=&\left\Vert f\right\Vert _{C^{k}\left( M,\mathbb{R}^{q}\right) }+\left[
\nabla ^{k}f\right] _{\alpha ,M;\mathbb{R}^{q}}.  \label{Schauder-norm}
\end{eqnarray}
\end{definition}

Then we have the following

\begin{lemma}
\label{product-inequ}Let $\cdot $ be the standard inner product in $\mathbb{R%
}^{q}$ (including $\ell ^{2}$). For $f$ and $g$ in $C^{k,\alpha }\left( M,%
\mathbb{R}^{q}\right) $, we have 
\begin{equation}
\left\Vert f\cdot g\right\Vert _{C^{k,\alpha }\left( M\right) }\leq C\left(
k,\alpha ,M\right) \left\Vert f\right\Vert _{C^{k,\alpha }\left( M,\mathbb{R}%
^{q}\right) }\left\Vert g\right\Vert _{C^{k,\alpha }\left( M,\mathbb{R}%
^{q}\right) },  \label{Cka}
\end{equation}%
where the constant $C\left( k,\alpha ,M\right) =n^{k}$ is uniform for all $q$%
.
\end{lemma}

\begin{proof}
The inequality already appeared in \cite{G1}. We adapt it to $\mathbb{R}^{q}$%
-valued functions and further observe that the constant $C\left( k,\alpha
,M\right) $ is \emph{uniform} for any dimensional $\mathbb{R}^{q}$. This is
because $D^{\overrightarrow{\gamma }}\left( f\cdot g\right) $ with $%
\left\vert \overrightarrow{\gamma }\right\vert \leq k$ produces at most $%
n^{k}$ inner product terms, and the \emph{Cauchy-Schwartz inequality} $%
\left\vert a\cdot b\right\vert _{\mathbb{R}^{q}}\leq \left\vert a\right\vert
_{\mathbb{R}^{q}}\left\vert b\right\vert _{\mathbb{R}^{q}}$ is valid for all 
$\mathbb{R}^{q}$ with coefficient $1$ on the right hand side.
\end{proof}

\subsection{G\"{u}nther's implicit function theorem \label{free-mapping}}

Given a free mapping $u:M\rightarrow \mathbb{R}^{q}$ (including $\ell ^{2}$,
if $q=\infty $) and $\left( h,f\right) \in C^{s,\alpha }\left( M,T^{\ast
}M\oplus \mathrm{Sym}^{\otimes 2}\left( T^{\ast }M\right) \right) $ with $%
s\geq 2$ and $0<\alpha <1.$ Let $v\left( x\right) \in C^{s,\alpha }\left( M,%
\mathbb{R}^{q}\right) $ $\ $be the unique solution to the following system 
\begin{equation}
P\left( u\right) \cdot v:=\left[ 
\begin{array}{c}
\nabla u \\ 
\nabla ^{2}u%
\end{array}%
\right] v=\left[ 
\begin{array}{c}
h \\ 
f%
\end{array}%
\right] \text{, and }v\left( x\right) \bot \ker P\left( u\right) \left(
x\right) ,  \label{linear-system}
\end{equation}%
where 
\begin{eqnarray*}
\nabla u &=&\left( \nabla u_{1},\cdots ,\nabla u_{q}\right) \in C^{s,\alpha
}\left( M,T^{\ast }M\otimes \mathbb{R}^{q}\right) , \\
\nabla ^{2}u &=&\left( \nabla ^{2}u_{1},\cdots ,\nabla ^{2}u_{q}\right) \in
C^{s,\alpha }\left( M,\mathrm{Sym}^{\otimes 2}\left( T^{\ast }M\right)
\otimes \mathbb{R}^{q}\right) , \\
P\left( u\right)  &:&C^{s,\alpha }\left( M,\mathbb{R}^{q}\right) \rightarrow
C^{s,\alpha }\left( M,T^{\ast }M\oplus \mathrm{Sym}^{\otimes 2}\left(
T^{\ast }M\right) \right) ,
\end{eqnarray*}%
and $\nabla u_{k}$ and $\nabla ^{2}u_{k}$ are the gradient and Hessian of $%
u_{k}$ respectively, for $1\leq k\leq q$ (Here we identify $T^{\ast }M\simeq
TM$ by the metric $g$, so $\nabla u_{k}$ can be regarded in $T^{\ast }M$).
If $u$ is a free mapping, then $\left( \ref{linear-system}\right) $ is
always solvable for any $\left( h,f\right) $. We define the right inverse of 
$P\left( u\right) $ as 
\begin{equation}
E\left( u\right) :%
\begin{array}{ccc}
C^{s,\alpha }\left( M,T^{\ast }M\oplus \mathrm{Sym}^{\otimes 2}\left(
T^{\ast }M\right) \right)  & \longrightarrow  & C^{s,\alpha }\left( M,%
\mathbb{R}^{q}\right)  \\ 
\left( h,f\right)  & \longmapsto  & v%
\end{array}%
.  \label{Eu}
\end{equation}%
By viewing $E\left( u\right) $ as a section of $C^{s,\alpha }\left( M,%
\mathrm{Hom}\left( T^{\ast }M\oplus \mathrm{Sym}^{\otimes 2}\left( T^{\ast
}M\right) ,\mathbb{R}^{q}\right) \right) $, the $C^{s,\alpha }\left(
M\right) $-norm of $E\left( u\right) $ is induced from the Riemannian and
Euclidean\ metrics $g$ and $g_{can}$. By linear algebra, there is an
explicit expression 
\begin{equation}
E\left( u\right) \left( x\right) =P^{T}\left( u\right) \left[ P\left(
u\right) P^{T}\left( u\right) \right] ^{-1}\left( x\right) ,
\label{Eu-expression}
\end{equation}%
where \textquotedblleft $T$\textquotedblright\ is the transpose. For an
orthonormal frame field $\left\{ V_{i}\right\} _{1\leq i\leq n}$ near $x$ on 
$M$, we have 
\begin{equation}
P\left( u\right) \left( x\right) \simeq \left[ 
\begin{array}{c}
\left\{ \nabla u\left( V_{i}\right) \left( x\right) \right\} ^{T} \\ 
\left\{ \nabla ^{2}u\left( V_{j,}V_{k}\right) \left( x\right) \right\} ^{T}%
\end{array}%
\right] _{1\,\leq i,j,k\leq n,j\leq k},  \label{Pu}
\end{equation}%
and $E\left( u\right) \left( x\right) $ is the unique \emph{right inverse}
of $P\left( u\right) \left( x\right) $ with its column vectors \emph{%
orthogonal} to $\ker P\left( u\right) \left( x\right) $. The $C^{s,\alpha
}\left( M\right) $-norm of $E\left( u\right) $ is the maximum of the $%
C^{s,\alpha }\left( M\right) $-norms of its column vectors, each viewed as a 
$\mathbb{R}^{q}$-valued function, defined in finitely many charts covering $M
$ by $\left( \ref{Schauder-norm}\right) $.

\begin{theorem}
\label{Gunther-Implicit-function-thm}(\cite{G2}) Let $u:\left(
M^{n},g\right) \rightarrow \mathbb{R}^{q}$ be a $C^{\infty }$-\emph{free
embedding }(cf. Definition \ref{free-map}). For $f\in C^{s,\alpha }\left( M,%
\mathrm{Sym}^{\otimes 2}\left( T^{\ast }M\right) \right) $ with $s\geq 2$ or 
$s=\infty $, and $0<\alpha <1$, there is a positive number $\theta $
(independent of $u,s$ and $f$) with the following property: If 
\begin{equation}
\left\Vert E\left( u\right) \right\Vert _{C^{2,\alpha }\left( M\right)
}\left\Vert E\left( u\right) \left( 0,f\right) \right\Vert _{C^{2,\alpha
}\left( M\right) }\leq \theta ,  \label{quadratic-error}
\end{equation}%
then there exists a $v=v\left( u,f\right) \in C^{s,\alpha }\left( M,\mathbb{R%
}^{q}\right) $ solving $d\left( u+v\right) \cdot d\left( u+v\right) =du\cdot
du+f$.
\end{theorem}

\begin{remark}
As will be seen in Section \ref{Qudratic-estimate} and Section \ref{IFT2}, G%
\"{u}nther's implicit function theorem has a \emph{uniform quadratic estimate%
} in $\mathbb{R}^{q}$ for all $q$, so the constant $\theta $ is independent
on $q$, including $q=\infty $.

To obtain the $v\in C^{s,\alpha }\left( M,\mathbb{R}^{q}\right) $ for $2\leq
s\leq \infty $, G\"{u}nther's theorem only requires to verify the $%
C^{2,\alpha }$-condition $\left( \ref{quadratic-error}\right) $. This makes
it easier to apply his theorem, but there appears no explicit control of the 
$C^{s,\alpha }$-norm of $v$, especially when $s>2$. For our purpose, it is
useful to know how close the isometric embedding map $I_{t}$ is to the
\textquotedblleft canonical\textquotedblright\ heat kernel embedding map $%
\Psi _{t}$ in the $C^{s,\alpha }$-norm for any given $s\geq 2$, so we will
prove a stronger inequality 
\begin{equation*}
\left\Vert E\left( u\right) \right\Vert _{C^{s,\alpha }\left( M\right)
}\left\Vert E\left( u\right) \left( 0,f\right) \right\Vert _{C^{s,\alpha
}\left( M\right) }\leq \theta 
\end{equation*}%
in our case, and give the estimate of $\left\Vert v\right\Vert _{C^{s,\alpha
}\left( M\right) }$.
\end{remark}

\subsection{G\"{u}nther's iteration scheme}

\bigskip The difficulty of applying the usual Banach space fixed point
theorem to system $\left( \ref{Perturbation-device}\right) $ lies in the
quadratic terms $\partial _{i}v\cdot \partial _{j}v$, which lose one order
of differentiability after each iteration of $v$. G\"{u}nther (\cite{G1})
invented a new iteration scheme with no loss of differentiability, which we
will recall in the following.

Let $\Delta _{\left( 1\right) }$ and $\Delta _{\left( 2\right) }$ be the 
\emph{Lichnerowicz Laplacian }for vector fields and symmetric covariant $2$%
-tensors on $\left( M,g\right) $ respectively, i.e. in local coordinates%
\begin{eqnarray*}
\Delta _{\left( 1\right) }t_{i} &:&=\Delta t_{i}-R_{i.}^{l}t_{l}\text{, }%
\Delta :=\nabla ^{l}\nabla _{l}, \\
\Delta _{\left( 2\right) }t_{ij} &:&=\Delta
t_{ij}-R_{i.j.}^{kl}t_{kl}-R_{i.}^{l}t_{ij}-R_{j.}^{l}t_{il}.
\end{eqnarray*}%
Fix a constant $\Lambda _{0}\not\in \mathrm{Spec}\left( \Delta _{\left(
r\right) }\right) $, the spectrum of $\Delta _{\left( r\right) }$ on $\left(
M,g\right) $ for $r=1,2$. We introduce the smoothing operator $\left( \Delta
_{\left( r\right) }-\Lambda _{0}\right) ^{-1}$, 
\begin{eqnarray*}
\left( \Delta _{\left( 1\right) }-\Lambda _{0}\right) ^{-1} &:&C^{s-2,\alpha
}\left( M,T^{\ast }M\right) \rightarrow C^{s,\alpha }\left( M,T^{\ast
}M\right) , \\
\left( \Delta _{\left( 2\right) }-\Lambda _{0}\right) ^{-1} &:&C^{s-2,\alpha
}\left( M,\mathrm{Sym}^{\otimes 2}\left( T^{\ast }M\right) \right)
\rightarrow C^{s,\alpha }\left( M,\mathrm{Sym}^{\otimes 2}\left( T^{\ast
}M\right) \right) ,
\end{eqnarray*}%
with the \emph{operator norm} denoted by $\left\Vert \left( \Delta _{\left(
r\right) }-\Lambda _{0}\right) ^{-1}\right\Vert _{\mathrm{op}}$, and let 
\begin{equation}
\sigma \left( \Lambda _{0},\alpha ,M\right) :=\max_{r=1,2}\left\Vert \left(
\Delta _{\left( r\right) }-\Lambda _{0}\right) ^{-1}\right\Vert _{\mathrm{op}%
}.  \label{spectral-constant}
\end{equation}%
Given a free mapping $u\in C^{s,\alpha }\left( M,\mathbb{R}^{q}\right) $ and 
$f\in C^{s,\alpha }\left( M,\mathrm{Sym}^{\otimes 2}\left( T^{\ast }M\right)
\right) $, let the vector field $N\left( v\right) $ and symmetric $2$-tensor
fields $L\left( v\right) $ and $M\left( v\right) $ be%
\begin{eqnarray}
N_{i}\left( v\right)  &=&-\Delta v\cdot \nabla _{i}v,  \notag \\
L_{ij}\left( v\right)  &=&2\nabla ^{l}\nabla _{i}v\cdot \nabla _{l}\nabla
_{j}v-2\Delta v\cdot \nabla _{i}\nabla _{j}v-2R_{i.j.}^{kl}\nabla _{k}v\cdot
\nabla _{l}v-\Lambda _{0}\nabla _{i}v\cdot \nabla _{j}v,  \notag \\
M_{ij}\left( v\right)  &=&\frac{1}{2}L_{ij}\left( v\right) +\left( \nabla
_{i}R_{j.}^{l}+\nabla _{j}R_{i.}^{l}-\nabla ^{l}R_{ij}\right) \left( \left(
\Delta _{\left( 1\right) }-\Lambda _{0}\right) ^{-1}N\left( v\right) \right)
_{l}  \label{NLM}
\end{eqnarray}%
in local coordinates, where the Einstein summation convention is used. G\"{u}%
nther defined the iteration $\Upsilon _{u}:C^{s,\alpha }\left( M,\mathbb{R}%
^{q}\right) \rightarrow C^{s,\alpha }\left( M,\mathbb{R}^{q}\right) $ by 
\begin{equation}
\Upsilon _{u}\left( v\right) :=E\left( u\right) \left( 0,-\frac{1}{2}%
f\right) +Q\left( u\right) \left( v,v\right) \text{,}
\label{Gunther-iteration}
\end{equation}%
where 
\begin{equation}
Q\left( u\right) \left( v,v\right) :=E\left( u\right) \left( \left( \Delta
_{\left( 1\right) }-\Lambda _{0}\right) ^{-1}N\left( v\right) ,\left( \Delta
_{\left( 2\right) }-\Lambda _{0}\right) ^{-1}M\left( v\right) \right) \in
C^{s,\alpha }\left( M,\mathbb{R}^{q}\right)   \label{smoothing-iteration}
\end{equation}%
For later reference, we denote the components 
\begin{eqnarray*}
Q_{i}\left( v\right)  &:&=\left( \left( \Delta _{\left( 1\right) }-\Lambda
_{0}\right) ^{-1}N\left( v\right) \right) _{i} \\
Q_{jk}\left( v\right)  &:&=\left( \left( \Delta _{\left( 2\right) }-\Lambda
_{0}\right) ^{-1}M\left( v\right) \right) _{jk}
\end{eqnarray*}%
for $1\leq i,j,k\leq n$.

\section{Uniform linear independence property of $P\left( \Psi _{t}\right) 
\label{Sec:uniform-indep}$}

\bigskip Recall that the (un-normalized) heat kernel embedding $\Phi
_{t}:\left( M,g\right) \rightarrow \ell ^{2}$ in \cite{BBG} is 
\begin{equation*}
\Phi _{t}:x\in M\rightarrow \left( e^{-\frac{\lambda _{1}}{2}t}\phi
_{1}\left( x\right) ,e^{-\frac{\lambda _{2}}{2}t}\phi _{2}\left( x\right)
,\cdots ,e^{-\frac{\lambda _{q}}{2}t}\phi _{q}\left( x\right) ,\cdots
\right) \in \ell ^{2}.
\end{equation*}%
For any $x\in M$ we take an orthonormal frame field $\left\{ V_{i}\right\}
_{1\leq i\leq n}$ in its neighborhood. Following our notation $P\left(
u\right) $ for a smooth map $u:M\rightarrow \mathbb{R}^{q}$, we consider the
following $\frac{n\left( n+3\right) }{2}\times \infty $ matrix $P\left( \Phi
_{t}\right) $ (with $q=\infty $) consisting of the $n$ first derivatives of $%
\Phi _{t}$ and the $n\left( n+1\right) /2$ second covariant derivatives
(i.e. Hessian) of $\Phi _{t}$ with respect to $\left\{ V_{i}\right\} _{1\leq
i\leq n}$: 
\begin{equation*}
P\left( \Phi _{t}\right) =\left[ 
\begin{array}{ccccc}
\cdots  & \cdots  & \cdots  & \cdots  & \cdots  \\ 
e^{-\frac{\lambda _{1}}{2}t}\nabla _{i}\phi _{1} & e^{-\frac{\lambda _{2}}{2}%
t}\nabla _{i}\phi _{2} & \cdots  & e^{-\frac{\lambda _{q}}{2}t}\nabla
_{i}\phi _{q} & \cdots  \\ 
\cdots  & \cdots  & \cdots  & \cdots  & \cdots  \\ 
e^{-\frac{\lambda _{1}}{2}t}\nabla _{j}\nabla _{k}\phi _{1} & e^{-\frac{%
\lambda _{2}}{2}t}\nabla _{j}\nabla _{k}\phi _{2} & \cdots  & e^{-\frac{%
\lambda _{q}}{2}t}\nabla _{j}\nabla _{k}\phi _{q} & \cdots  \\ 
\cdots  & \cdots  & \cdots  & \cdots  & \cdots 
\end{array}%
\right] .
\end{equation*}

We will prove as $t\rightarrow 0_{+}$ , $P\left( \Phi _{t}\right) \left(
x\right) $ has full rank for any $x$ on $M$, so $\Phi _{t}$ is a free
mapping. For this we will compute the inner products of the row vectors of $%
P\left( \Phi _{t}\right) $. The following \emph{uniform freeness} result of $%
\Phi _{t}$ will be proved in this section.

\begin{theorem}
\label{Uni-indep}(Uniform linear independence) As $t\rightarrow 0_{+}$, we
have the following asymptotic formulae 
\begin{equation*}
\frac{\left\langle \nabla _{i}\Phi _{t},\nabla _{j}\Phi _{t}\right\rangle }{%
\left\vert \nabla _{i}\Phi _{t}\right\vert \left\vert \nabla _{j}\Phi
_{t}\right\vert }=\delta _{ij}+O\left( t\right) \text{, \ }\frac{%
\left\langle \nabla _{i}\nabla _{j}\Phi _{t},\nabla _{k}\Phi
_{t}\right\rangle }{\left\vert \nabla _{i}\nabla _{j}\Phi _{t}\right\vert
\left\vert \nabla _{k}\Phi _{t}\right\vert }=O\left( t\right) ,
\end{equation*}%
and 
\begin{equation*}
\frac{\left\langle \nabla _{i}\nabla _{j}\Phi _{t},\nabla _{k}\nabla
_{l}\Phi _{t}\right\rangle }{\left\vert \nabla _{i}\nabla _{j}\Phi
_{t}\right\vert \left\vert \nabla _{k}\nabla _{l}\Phi _{t}\right\vert }%
=O\left( t\right) +\left\{ 
\begin{array}{cc}
1, & \left\{ i,j\right\} =\left\{ k,l\right\} \text{ as sets} \\ 
1/3, & i=j\text{ and }k=l\text{, but }i\neq k \\ 
0, & \text{otherwise}%
\end{array}%
,\right. 
\end{equation*}%
where $\left\langle ,\right\rangle $ is the standard inner product in $\ell
^{2}$. The above convergence is uniform for all $x$ on $M$ in the $C^{r}$%
-norm for any $r\geq 0$. Moreover, if we truncate $\Phi _{t}:M\rightarrow 
\mathbb{R}^{q}\subset \ell ^{2}$ for $q=q\left( t\right) \geq t^{-\frac{n}{2}%
-\rho }$ with sufficiently small $t>0$, the above results still
hold.\bigskip 
\end{theorem}

\subsection{The Minakshisundaram-Pleijel expansion}

As in \cite{BBG}, we have the \emph{Minakshisundaram-Pleijel expansion}$\,$\
of the heat kernel 
\begin{equation}
H\left( t,x,y\right) =\frac{1}{\left( 4\pi t\right) ^{n/2}}e^{-\frac{r^{2}}{%
4t}}U\left( t,x,y\right) ,  \label{MP-expansion}
\end{equation}%
where $r=r\left( x,y\right) $ is the distance function for points $x$ and $y$
on $M$, and 
\begin{equation}
U\left( t,x,y\right) =u_{0}\left( x,y\right) +tu_{1}\left( x,y\right)
+\cdots +t^{p}u_{p}\left( x,y\right) +O\left( t^{p+1}\right) ,
\label{U-expansion}
\end{equation}%
in the $C^{r}$ sense for any $r\geq 0$ (see \cite{BeGaM} p. 213).

It is known $u_{0}\left( x,y\right) =\left[ \theta \left( x,y\right) \right]
^{-1/2}$ (for $x$ and $y$ close enough), and 
\begin{equation}
\theta \left( x,y\right) =\frac{\text{volume density at }y\text{ read in the
normal coordinate around }x}{r^{n-1}}  \label{theta_xy}
\end{equation}%
with $r=r\left( x,y\right) $ (\cite{BeGaM}, p. 208). Especially 
\begin{equation}
\theta \left( x,x\right) =1=u_{0}\left( x,x\right) .  \label{u0xx}
\end{equation}%
From this we immediately see 
\begin{equation}
\left\langle \Phi _{t},\Phi _{t}\right\rangle \left( x\right) =H\left(
t,x,x\right) =\frac{1}{\left( 4\pi t\right) ^{n/2}}\left( 1+O\left( t\right)
\right) .  \label{phi_t_norm}
\end{equation}%
We also have%
\begin{equation}
\partial _{i}U|_{x=y}=\partial _{i}u_{0}\left( x,y\right) |_{x=y}+t\partial
_{i}u_{1}\left( x,y\right) |_{x=y}+O\left( t^{2}\right) =O\left( t\right) 
\text{,}  \label{diU}
\end{equation}%
where $\partial _{i}u_{0}\left( x,y\right) |_{x=y}=0$ in $\left( \ref{diU}%
\right) $ is due to 
\begin{eqnarray*}
u_{0}\left( x,0\right) &=&\left[ \theta \left( x\left( s\right) ,0\right) %
\right] ^{-\frac{1}{2}}=\left[ 1-\text{Ric}\left( \dot{x}\left( s\right) ,%
\dot{x}\left( s\right) \right) \frac{s^{2}}{3!}+O\left( \left\vert
s\right\vert ^{3}\right) \right] ^{-\frac{1}{2}} \\
&=&1+\frac{1}{12}\text{Ric}\left( \dot{x}\left( s\right) ,\dot{x}\left(
s\right) \right) s^{2}+O\left( \left\vert s\right\vert ^{3}\right) .
\end{eqnarray*}

We recall the following useful lemma.

\begin{lemma}
\label{Lem:r} Let $r=r\left( x,y\right) $ be the shortest distance between $x
$ and $y$ on $M$. For $x$ and $y$ which are close enough to each other, $%
r:M\times M\rightarrow \mathbb{R}$ is smooth. Using the normal coordinates $%
\left\{ x^{i}\right\} _{1\leq i\leq n}$ near $x$, we have 
\begin{eqnarray}
\partial _{i}^{x}r^{2}\left( x,y\right) |_{x=y}\text{ } &=&\partial
_{i}^{y}r^{2}\left( x,y\right) |_{x=y}=0,  \label{1st-deri-r2} \\
\partial _{i}^{x}\partial _{j}^{x}r^{2}\left( x,y\right) |_{x=y}
&=&-\partial _{i}^{x}\partial _{j}^{y}r^{2}\left( x,y\right) |_{x=y}=2\delta
_{ij},  \label{r-derivatives} \\
\partial _{k}^{x}\partial _{i}^{x}\partial _{j}^{x}r^{2}\left( x,y\right)
|_{x=y} &=&\partial _{k}^{y}\partial _{i}^{x}\partial _{j}^{x}r^{2}\left(
x,y\right) |_{x=y}=0,  \label{third-deri-r2}
\end{eqnarray}%
where the notation $\partial _{i}^{x}$ (resp. $\partial _{i}^{y}$) means the
derivative is taken with respect to the variable in the first (resp. second)
component of $M\times M$. 
\end{lemma}

The first two identities are proved in \cite{BBG}. The third identity can be
computed in the normal coordinates near $x$ (see e.g. \cite{De}, Chapter 16,
p.282), using the Taylor expansion of the metric $g$ near $x$ (e.g. \cite{T}
Proposition 3.1, p. 41) 
\begin{eqnarray*}
g_{ij}(x) &=&\delta _{ij}+\frac{1}{3}R_{iklj}x^{k}x^{l}+\frac{1}{6}%
R_{iklj,s}x^{k}x^{l}x^{s} \\
&&+\left( \frac{1}{20}R_{iklj,st}+\frac{2}{45}\Sigma
_{m}R_{iklm}R_{jstm}\right) x^{k}x^{l}x^{s}x^{t}+O\left( r^{5}\right) ,
\end{eqnarray*}%
where $r$ is the distance to the base point $x_{0}$. 

\subsection{Derivative estimates of the heat kernel embedding $\Psi _{t}$}

\bigskip Using Lemma \ref{Lem:r} and the Minakshisundaram-Pleijel expansion $%
\left( \ref{MP-expansion}\right) $ we will derive higher derivative
estimates of $D^{\overrightarrow{\alpha }}\Phi _{t}\left( x\right) $. Let $%
\overrightarrow{\gamma }=\left( \overrightarrow{\alpha },\overrightarrow{%
\beta }\right) $ be a multi-index, with $\overrightarrow{\alpha }$ and $%
\overrightarrow{\beta }$ being the multi-indices in $x$ and $y$ variables of 
$M\times M$ respectively. Let $D^{\overrightarrow{\gamma }}$ be the
corresponding multi-derivative operator. From the heat kernel expression $%
H\left( t,x,y\right) =\Sigma _{s=1}^{\infty }e^{-\lambda _{s}t}\phi
_{s}\left( x\right) \phi _{s}\left( y\right) $, it is easy to check%
\begin{equation}
\left\langle D^{\overrightarrow{\alpha }}\Phi _{t}\left( x\right) ,D^{%
\overrightarrow{\beta }}\Phi _{t}\left( x\right) \right\rangle =D^{%
\overrightarrow{\gamma }}H\left( t,x,y\right) |_{x=y}\text{.}
\label{DaDbDg_H}
\end{equation}

\begin{proposition}
\label{high-derivative} As $t\rightarrow 0_{+}$, there exists a constant $C>0
$ such that%
\begin{eqnarray}
\left\vert D^{\overrightarrow{\gamma }}H\left( t,x,y\right)
|_{x=y}\right\vert  &\leq &Ct^{-\frac{n}{2}-\left[ \frac{\left\vert 
\overrightarrow{\gamma }\right\vert }{2}\right] },  \label{coarse-asymp} \\
\left\vert D^{\overrightarrow{\alpha }}\Phi _{t}\left( x\right) \right\vert
^{2} &\leq &Ct^{-\frac{n}{2}-\left\vert \overrightarrow{\alpha }\right\vert
},  \notag
\end{eqnarray}%
where $\left[ b\right] $ is the largest integer less or equal to a given
real number $b$.
\end{proposition}

\begin{proof}
We write $H\left( t,x,y\right) =\frac{1}{\left( 4\pi t\right) ^{n/2}}e^{-%
\frac{r^{2}}{4t}}U\left( t,x,y\right) $, and 
\begin{equation}
D^{\overrightarrow{\gamma }}H\left( t,x,y\right) =\frac{1}{\left( 4\pi
t\right) ^{n/2}}e^{-\frac{r^{2}}{4t}}P_{\overrightarrow{\gamma }}\left(
t,x,y\right) ,  \label{Pg_txy}
\end{equation}%
where $P_{\overrightarrow{\gamma }}\left( t,x,y\right) $ is a \emph{%
polynomial} in $D^{\overrightarrow{\mu _{j}}}\left( \frac{r^{2}\left(
x,y\right) }{t}\right) $ and $D^{\overrightarrow{\eta _{k}}}U\left(
t,x,y\right) $ for multi-indices $\overrightarrow{\mu _{j}}$ and $%
\overrightarrow{\eta _{k}}$ by the Leibniz rule. For example, when $%
\overrightarrow{\gamma }=\partial _{x_{i}},$%
\begin{equation*}
P_{\overrightarrow{i}}\left( t,x,y\right) =-\frac{1}{4}\partial _{i}\left( 
\frac{r^{2}\left( x,y\right) }{t}\right) U\left( t,x,y\right) +\partial
_{i}U\left( t,x,y\right) .
\end{equation*}%
We have

1. Each summand of $P_{\overrightarrow{\gamma }}\left( x,t\right) $ is of
the form%
\begin{equation}
\left. \left( \Pi _{j=1}^{s}D^{\overrightarrow{\mu _{j}}}\left( \frac{%
r^{2}\left( x,y\right) }{t}\right) \right) \right\vert _{x=y}\cdot \left. D^{%
\overrightarrow{\eta }}U\left( t,x,y\right) \right\vert _{x=y}
\label{summand}
\end{equation}%
with 
\begin{equation}
\Sigma _{j=1}^{s}\left\vert \overrightarrow{\mu _{j}}\right\vert +\left\vert 
\overrightarrow{\eta }\right\vert =\left\vert \overrightarrow{\gamma }%
\right\vert .  \label{total-degree}
\end{equation}%
There are only finitely many terms in $P_{\overrightarrow{\gamma }}\left(
x,t\right) $; we denote the total number by $n\left( \overrightarrow{\gamma }%
\right) $.

2. As $t\rightarrow 0_{+}$, the terms involving the highest possible power
of $\frac{1}{t}$ must have $\overrightarrow{\mu _{j}}^{\prime }s$ with $%
\left\vert \overrightarrow{\mu _{j}}\right\vert =2$ as many as possible.
This is because of the total degree condition $\left( \ref{total-degree}%
\right) $, and Lemma \ref{Lem:r}. So if there are more than one $%
\overrightarrow{\mu _{j}}$ with $\left\vert \overrightarrow{\mu _{j}}%
\right\vert \geq 3$, the summand $\left( \ref{summand}\right) $ loses the
potential to have the maximal number of factors $\frac{1}{t}$, which is $%
\left[ \frac{\left\vert \overrightarrow{\gamma }\right\vert }{2}\right] $ by 
$\left( \ref{total-degree}\right) $.

Therefore from $\left( \ref{Pg_txy}\right) $ we have 
\begin{equation*}
\left\vert D^{\overrightarrow{\gamma }}H\left( t,x,y\right)
|_{x=y}\right\vert \leq C\frac{1}{\left( 4\pi t\right) ^{n/2}}\cdot t^{-%
\left[ \frac{\left\vert \overrightarrow{\gamma }\right\vert }{2}\right] },
\end{equation*}%
\begin{equation*}
\left\vert D^{\overrightarrow{\alpha }}\Psi _{t}\left( x\right) \right\vert
^{2}=\left\vert D_{x}^{\overrightarrow{\alpha }}D_{y}^{\overrightarrow{%
\alpha }}H\left( t,x,y\right) |_{x=y}\right\vert \leq C\frac{1}{\left( 4\pi
t\right) ^{n/2}}\cdot t^{-\left[ \frac{2\left\vert \overrightarrow{\alpha }%
\right\vert }{2}\right] }=C\frac{1}{\left( 4\pi t\right) ^{n/2}}\cdot
t^{-\left\vert \overrightarrow{\alpha }\right\vert }.
\end{equation*}

The constant $C$ can be taken as $C=n\left( \overrightarrow{\alpha }\right)
\sup \left\Vert U\left( t,x,y\right) \right\Vert _{C^{\left\vert 
\overrightarrow{\alpha }\right\vert }\left( M\times M\right) }$, where the $%
\sup $ is taken for $t>0$ in the range of the Minakshisundaram-Pleijel
expansion $\left( \ref{MP-expansion}\right) $, and $\left( x,y\right) \in
M\times M$ with dist$\left( x,y\right) $ less than the injective radius of $%
M $.
\end{proof}

We can get precise asymptotic formulae if we are more careful in item 2 of
the above argument, as the following

\begin{proposition}
\label{heat-asymp} For any $x$ on $M$, let $\left\{ x^{i}\right\} _{i=1}^{n}$
be the normal coordinates near $x$. Then as $t\rightarrow 0_{+}$, we have 
\begin{eqnarray*}
\text{ }\left\langle \partial _{i}\Phi _{t},\partial _{j}\Phi
_{t}\right\rangle \left( x\right)  &=&\left( 2t\right) ^{-1}\left( 4\pi
t\right) ^{-n/2}\left( \delta _{ij}+O\left( t\right) \right) ,\text{ } \\
\left\langle \partial _{i}\partial _{j}\Phi _{t},\partial _{k}\Phi
_{t}\right\rangle \left( x\right)  &=&t^{-n/2}\cdot O\left( 1\right) ,\text{ 
}
\end{eqnarray*}%
\begin{equation}
\left\langle \partial _{j}\partial _{i}\Phi _{t},\partial _{m}\partial
_{k}\Phi _{t}\right\rangle \left( x\right) =\left( 2t\right) ^{-2}\left(
4\pi t\right) ^{-n/2}\left( \delta _{ij}\delta _{km}+\delta _{im}\delta
_{jk}+\delta _{ik}\delta _{jm}+O\left( t\right) \right)   \label{4th-neat}
\end{equation}%
for $1\leq i,j,k,m\leq n$. The convergence is uniform for all $x$ on $M$ in
the $C^{r}$-norm for any $r\geq 0$. The coefficients of $O\left( t\right) $
and $O\left( 1\right) $ are determined by the covariant differentiations of
the metric $g$ and its curvature tensors.
\end{proposition}

\begin{proof}
The first asymptotic formula is in \cite{BBG}. We will prove the remaining
two. For simplicity we let $\partial _{i}^{x}=\partial _{i}$ and $\partial
_{k}^{y}=\partial _{\bar{k}}$ from now on (the bar notation in
\textquotedblleft $\overline{k}$\textquotedblright\ means that the
derivative is taken with respect to the $y$ variable for $\left( x,y\right)
\in M\times M$). \ We will compute $D^{\overrightarrow{\gamma }}H\left(
t,x,y\right) |_{x=y}$ for multi-indices $\overrightarrow{\gamma }$ in the
following two cases, and observe the summands $\left( \ref{summand}\right) $
in $P_{\overrightarrow{\gamma }}\left( t,x,y\right) $ $\left( \ref{Pg_txy}%
\right) $ of leading order:

\begin{enumerate}
\item $\overrightarrow{\gamma }=\left( \partial _{j}\partial _{i},\partial _{%
\bar{k}}\right) $: By $\left( \ref{total-degree}\right) $, $\Sigma
_{j=1}^{s}\left\vert \overrightarrow{\mu _{j}}\right\vert +\left\vert 
\overrightarrow{\eta }\right\vert =3$. So the highest possible power of $%
\left( \frac{1}{t}\right) $ in $P_{\overrightarrow{\gamma }}\left(
t,x,y\right) $ is $\left[ \frac{3}{2}\right] =1$, with $s=1$, $\left\vert 
\overrightarrow{\mu _{1}}\right\vert =2$, $\left\vert \overrightarrow{\eta }%
\right\vert =1$. The term in $P_{\overrightarrow{\gamma }}\left(
t,x,y\right) $ containing $\frac{1}{t}$ is 
\begin{eqnarray*}
&&\left. \partial _{\overline{k}}\partial _{j}\left( -\frac{r^{2}}{4t}%
\right) \partial _{i}U+\partial _{j}\partial _{i}\left( -\frac{r^{2}}{4t}%
\right) \partial _{\overline{k}}U+\partial _{\overline{k}}\partial
_{i}\left( -\frac{r^{2}}{4t}\right) \partial _{j}U\right\vert _{x=y} \\
&=&-\frac{1}{2t}\left( -\delta _{kj}\partial _{i}U+\delta _{ij}\partial _{%
\overline{k}}U-\delta _{ik}\partial _{j}U\right) =O\left( 1\right) \text{, }
\end{eqnarray*}%
where we have used $\left( \ref{diU}\right) $ that $\nabla U|_{x=y}=O\left(
t\right) $. Since we have used $\nabla U|_{x=y}=O\left( t\right) $ to
decrease the $\left( \frac{1}{t}\right) $ power by $1$, we must also
consider the term in $P_{\overrightarrow{\gamma }}\left( t,x,y\right) $ with 
$s=0$ and $\left\vert \overrightarrow{\eta }\right\vert =3$, but clearly $%
\partial _{j}\partial _{i}\partial _{\bar{k}}U|_{x=y}=O\left( 1\right) $.
Hence$\left\vert \left\langle \partial _{i}\partial _{j}\Phi _{t},\partial
_{k}\Phi _{t}\right\rangle \left( x\right) \right\vert =\left. \partial _{%
\bar{k}}\partial _{j}\partial _{i}H\left( t,x,y\right) \right\vert
_{x=y}=t^{-n/2}\cdot O\left( 1\right) .$

\item $\overrightarrow{\gamma }=\left( \partial _{j}\partial _{i},\partial _{%
\bar{m}}\partial _{\bar{k}}\right) $: By $\left( \ref{total-degree}\right) $%
, $\Sigma _{j=1}^{s}\left\vert \overrightarrow{\mu _{j}}\right\vert
+\left\vert \overrightarrow{\eta }\right\vert =4$. So the highest possible
power of $\left( \frac{1}{t}\right) $ in $P_{\overrightarrow{\gamma }}\left(
t,x,y\right) $ is $\left[ \frac{4}{2}\right] =2$, with $s=2$, $\left\vert 
\overrightarrow{\mu _{1}}\right\vert =\left\vert \overrightarrow{\mu _{2}}%
\right\vert =2$,$\left\vert \overrightarrow{\eta }\right\vert =0$. The term
in $P_{\overrightarrow{\gamma }}\left( t,x,y\right) $ containing $\left( 
\frac{1}{t}\right) ^{2}$ is%
\begin{eqnarray*}
&&\left. \left( -\frac{1}{4t}\right) ^{2}\left[ \partial _{\bar{m}}\partial
_{\bar{k}}\left( r^{2}\right) \partial _{j}\partial _{i}\left( r^{2}\right)
+\partial _{\overline{k}}\partial _{j}\left( r^{2}\right) \partial _{\bar{m}%
}\partial _{i}\left( r^{2}\right) +\partial _{\bar{m}}\partial _{j}\left(
r^{2}\right) \partial _{\overline{k}}\partial _{i}\left( r^{2}\right) \right]
U\right\vert _{x=y} \\
&=&\left( \frac{1}{2t}\right) ^{2}\left( \delta _{ij}\delta _{km}+\delta
_{im}\delta _{jk}+\delta _{ik}\delta _{jm}+O\left( t\right) \right) .
\end{eqnarray*}%
Hence%
\begin{eqnarray*}
&&\left\langle \partial _{j}\partial _{i}\Phi _{t},\partial _{m}\partial
_{k}\Phi _{t}\right\rangle \left( x\right) =\partial _{\bar{m}}\partial _{%
\bar{k}}\partial _{j}\partial _{i}H\left( t,x,y\right) |_{x=y} \\
&=&\left( \frac{1}{2t}\right) ^{2}\frac{1}{\left( 4\pi t\right) ^{n/2}}%
\left( \delta _{ij}\delta _{km}+\delta _{im}\delta _{jk}+\delta _{ik}\delta
_{jm}+O\left( t\right) \right) .
\end{eqnarray*}
\end{enumerate}

The coefficients of $O\left( t\right) $ and $O\left( 1\right) $ in the above
argument are determined by the covariant differentiations of the metric $g$
and its curvature tensors, because of $\left( \ref{DaDbDg_H}\right) $ and
the heat kernel expansion of $D^{\overrightarrow{\gamma }}H\left(
t,x,y\right) |_{x=y}$, similar to the expansion of $H\left( t,x,x\right) $
(e.g. \cite{Gil}).
\end{proof}

\begin{remark}
In the expansion of $\partial _{\bar{m}}\partial _{\bar{k}}\partial
_{j}\partial _{i}H\left( t,x,y\right) $, the term $\frac{\partial _{\bar{m}%
}\partial _{\overline{k}}\partial _{j}\partial _{i}\left( r^{2}\right) }{4t}U
$ involves the \emph{curvature tensors} on $M$ (see e.g. \cite{De}, Chapter
16, p.282), but it is a lower order term (order $t^{-1}$ v.s. leading order $%
t^{-2}$) so does not affect the asymptotic behavior.
\end{remark}

\begin{remark}
\label{Ni_work}Proposition \ref{high-derivative} and \ref{heat-asymp} with
precise coefficients $C$ (such that the inequalities become equalities as $%
t\rightarrow 0_{+}$) was also obtained in \cite{Ni} in the context of random
function theory by different argument. For the purpose of our paper we do
not need that general result, and prefer to give a self-contained
derivation. \cite{Ni} also proved the almost isometric embeddings by
eigenfunctions and a wide class of weights. Potentially some of them may be
perturbed to isometric embeddings by our method. 
\end{remark}

Later we will need the H\"{o}lder derivative estimate of $D^{\overrightarrow{%
\alpha }}\Psi _{t}\left( x\right) $. This can be obtained by interpolation
between integer derivative estimates in the above Proposition, similar to
the argument in Lemma \ref{eigen-deri-est}, by distinguish the two cases
that dist$\left( x,y\right) $\thinspace $\leq t^{\frac{1}{2}}$ and dist$%
\left( x,y\right) $\thinspace $\geq t^{\frac{1}{2}}$ in estimating the term $%
\frac{D^{\overrightarrow{\alpha }}\Phi _{t}\left( x\right) -D^{%
\overrightarrow{\alpha }}\Phi _{t}\left( y\right) }{\left( \text{dist}\left(
x,y\right) \right) ^{\alpha }}$. We have

\begin{proposition}
\label{Schauder-derivative-est}As $t\rightarrow 0_{+}$, the H\"{o}lder
derivatives satisfy 
\begin{eqnarray*}
\left[ D^{\overrightarrow{\alpha }}\Phi _{t}\left( x\right) \right] _{\alpha
;M} &\leq &Ct^{-\frac{n}{4}-\frac{\left\vert \overrightarrow{\alpha }%
\right\vert +\alpha }{2}}, \\
\left\Vert \Phi _{t}\left( x\right) \right\Vert _{C^{k,\alpha }\left(
M\right) } &\leq &Ct^{-\frac{n}{4}-\frac{k+\alpha }{2}}
\end{eqnarray*}%
\bigskip for some constant $C>0$.
\end{proposition}

\subsection{\protect\bigskip Uniform linear independence property of $%
P\left( \Psi _{t}\right) \label{subsec:uni-indep}$}

\begin{proof}[Proof of Theorem \protect\ref{Uni-indep}]
For any $x\in M$, we choose the normal coordinates $\left\{ x^{i}\right\}
_{1\leq i\leq n}$ near $x$ such that $\left\{ \frac{\partial }{\partial x^{i}%
}\right\} _{1\leq i\leq n}$ agree with the frame field $\left\{
V_{i}\right\} _{1\leq i\leq n}$ at $x$. Then $\nabla _{i}\Phi _{t}\left(
x\right) =\partial _{i}\Phi _{t}\left( x\right) $, $\nabla _{j}\nabla
_{k}\Phi _{t}\left( x\right) =\partial _{j}\partial _{k}\Phi _{t}\left(
x\right) $. From Proposition \ref{heat-asymp} we have as $t\rightarrow 0_{+}$%
, 
\begin{equation}
\left\vert \nabla _{i}\Phi _{t}\left( x\right) \right\vert ^{2}=\left(
2t\right) ^{-1}\left( 4\pi t\right) ^{-n/2}\cdot \left( 1+O\left( t\right)
\right) \text{, and }\frac{\left\langle \nabla _{i}\Phi _{t},\nabla _{j}\Phi
_{t}\right\rangle }{\left\vert \nabla _{i}\Phi _{t}\right\vert \left\vert
\nabla _{j}\Phi _{t}\right\vert }\left( x\right) =\delta _{ij}+O\left(
t\right) .  \label{ortho_ij}
\end{equation}%
From $\left( \ref{4th-neat}\right) $ we have 
\begin{equation*}
\left\langle \nabla _{j}\nabla _{i}\Phi _{t},\nabla _{m}\nabla _{k}\Phi
_{t}\right\rangle \left( x\right) =\left( 2t\right) ^{-2}\left( 4\pi
t\right) ^{-n/2}\cdot \left( \delta _{ij}\delta _{km}+\delta _{im}\delta
_{jk}+\delta _{ik}\delta _{jm}+O\left( t\right) \right) .
\end{equation*}%
Especially for $i\neq j$ and $\left\{ k,m\right\} \neq \left\{ i,j\right\} $
as sets, we have 
\begin{eqnarray}
\left\langle \nabla _{i}\nabla _{i}\Phi _{t},\nabla _{i}\nabla _{i}\Phi
_{t}\right\rangle \left( x\right)  &=&\left( 2t\right) ^{-2}\left( 4\pi
t\right) ^{-n/2}\cdot \left( 3+O\left( t\right) \right) ,\text{ }  \notag \\
\left\langle \nabla _{j}\nabla _{i}\Phi _{t},\nabla _{j}\nabla _{i}\Phi
_{t}\right\rangle \left( x\right)  &=&\left( 2t\right) ^{-2}\left( 4\pi
t\right) ^{-n/2}\cdot \left( 1+O\left( t\right) \right) ,
\label{ijij-length} \\
\left\langle \nabla _{i}\nabla _{i}\Phi _{t},\nabla _{j}\nabla _{j}\Phi
_{t}\right\rangle \left( x\right)  &=&\left( 2t\right) ^{-2}\left( 4\pi
t\right) ^{-n/2}\cdot \left( 1+O\left( t\right) \right) ,  \notag \\
\left\langle \nabla _{j}\nabla _{i}\Phi _{t},\nabla _{m}\nabla _{k}\Phi
_{t}\right\rangle \left( x\right)  &=&\left( 2t\right) ^{-2}\left( 4\pi
t\right) ^{-n/2}\cdot \left( 0+O\left( t\right) \right) ,\text{ }  \notag
\end{eqnarray}%
So we conclude that 
\begin{equation}
\frac{\left\langle \nabla _{i}\nabla _{j}\Phi _{t},\nabla _{k}\nabla
_{l}\Phi _{t}\right\rangle }{\left\vert \nabla _{i}\nabla _{j}\Phi
_{t}\right\vert \left\vert \nabla _{k}\nabla _{l}\Phi _{t}\right\vert }%
\left( x\right) =\left\{ 
\begin{tabular}{ll}
$0+O\left( t\right) \text{,}$ & if $\left\{ i,j\right\} \neq \left\{
k,l\right\} $ and $\left\{ i,k\right\} \neq \left\{ j,l\right\} $ as sets,
\\ 
$1/3+O\left( t\right) \text{,}$ & $\text{if }i=j\text{ and }k=l\text{ , but }%
i\neq k,$ \\ 
$1+O\left( t\right) \text{,}$ & if $\left\{ i,j\right\} =\left\{ k,l\right\} 
$ as sets.%
\end{tabular}%
\right.   \label{ortho-ijkl}
\end{equation}%
By $\left( \ref{ijij-length}\right) $ $\left\vert \nabla _{j}\nabla _{i}\Phi
_{t}\right\vert ^{2}\rightarrow Ct^{-\frac{n}{2}-2}$, combining Proposition %
\ref{heat-asymp} we have 
\begin{eqnarray}
\frac{\left\vert \left\langle \nabla _{i}\nabla _{j}\Phi _{t},\nabla
_{k}\Phi _{t}\right\rangle \right\vert }{\left\vert \nabla _{i}\nabla
_{j}\Phi _{t}\right\vert \left\vert \nabla _{k}\Phi _{t}\right\vert }\left(
x\right)  &=&\frac{t^{-n/2}\cdot O\left( 1\right) }{\left[ \left( 2t\right)
^{-2}\left( 4\pi t\right) ^{-n/2}\cdot \left( 1+O\left( t\right) \right)
\cdot \left( 2t\right) ^{-1}\left( 4\pi t\right) ^{-n/2}\cdot \left(
1+O\left( t\right) \right) \right] ^{1/2}}  \notag \\
&=&O\left( t^{3/2}\right) .  \label{ortho-ijk}
\end{eqnarray}%
By Proposition \ref{isom-truncation-high-jet}, the above inner product
results pass to the truncated map $\Phi _{t}:\left( M,g\right) \rightarrow 
\mathbb{R}^{q\left( t\right) }$ as well.

The linear independence of the row vectors follows from $\left( \ref%
{ortho_ij}\right) ,\left( \ref{ortho-ijkl}\right) ,\left( \ref{ortho-ijk}%
\right) $ and taking 
\begin{equation*}
\alpha _{i}=\nabla _{i}\nabla _{i}\Phi _{t}\left( x\right) \text{ for }%
i=1,\cdots ,n
\end{equation*}%
in the following linear algebra lemma.
\end{proof}

\begin{lemma}
\label{lin-alg}For $n$ vectors $\alpha _{1},\cdots ,\alpha _{n}$ in a real
linear space $V$ equipped with an inner product $\left\langle ,\right\rangle 
$ , if there is a constant $\sigma \in \left( -\frac{1}{n-1},1\right) $ such
that 
\begin{equation*}
\frac{\left\langle \alpha _{i},\alpha _{j}\right\rangle }{\left\vert \alpha
_{i}\right\vert \left\vert \alpha _{j}\right\vert }=\sigma ,\text{ for all }%
i\neq j\text{.}
\end{equation*}%
then $\left\{ \alpha _{i}\right\} _{1\leq i\leq n}$ are linearly independent.
\end{lemma}

\begin{proof}
We will find an explicit invertible linear transform $L_{0}:V\rightarrow V$
to change the angle $\cos ^{-1}\sigma $ between $\alpha _{i}$ and $\alpha
_{j}$ to $\frac{\pi }{2}$. With out loss of generality we can assume all $%
\alpha _{i}$ are unit vectors. Let 
\begin{eqnarray*}
\alpha _{0} &=&\frac{\alpha _{1}+\cdots +\alpha _{n}}{n},\text{ }c=c\left(
\sigma ,n\right) =\sqrt{\frac{1+\left( n-1\right) \sigma }{1-\sigma }}\text{
(note }\sigma >-\frac{1}{n-1}\text{),} \\
\widetilde{\alpha _{i}} &=&\alpha _{0}+c\left( \alpha _{i}-\alpha
_{0}\right) :=L_{0}\alpha _{i}\text{, for }i=1,\cdots ,n,
\end{eqnarray*}%
then $\widetilde{\alpha _{i}}$ is nonzero by checking the orthogonal
relation $\left\langle \alpha _{0},\alpha _{i}-\alpha _{0}\right\rangle =0$.
We also have for $i\neq j$, 
\begin{eqnarray*}
\left\langle \widetilde{\alpha _{i}},\widetilde{\alpha _{j}}\right\rangle
&=&\left\vert \alpha _{0}\right\vert ^{2}+c^{2}\left( \alpha _{i}-\alpha
_{0}\right) \left( \alpha _{j}-\alpha _{0}\right) \\
&=&\frac{\left( n-1\right) \sigma +1}{n}+\frac{1+\left( n-1\right) \sigma }{%
1-\sigma }\left( \sigma -\frac{\left( n-1\right) \sigma +1}{n}\right) =0,
\end{eqnarray*}%
so $\left\{ \widetilde{\alpha _{i}}\right\} _{1\leq i\leq n}$ is an
orthogonal set. Since $\left\{ \widetilde{\alpha _{i}}\right\} _{1\leq i\leq
n}$ is obtained from linear combinations of $\left\{ \alpha _{i}\right\}
_{1\leq i\leq n}$, $\left\{ \alpha _{i}\right\} _{1\leq i\leq n}$ must be
linearly independent.
\end{proof}

\begin{corollary}
\label{Theta-matrix}Let $\sigma $ $\in \left( -\frac{1}{n-1},1\right) $.
Then the $n\times n$ matrix 
\begin{equation}
\Xi _{n}\left( \sigma \right) :=\left[ \theta _{ij}\right] _{1\leq i,j\leq n}
\label{Angle-matrix}
\end{equation}%
with $\theta _{ii}=1$ and $\theta _{ij}=\sigma $ $\left( i\neq j\right) $ is
invertible.
\end{corollary}

\begin{proof}
Let $P$ be the matrix whose row vectors are the above unit vectors $\alpha
_{i}$ $\left( 1\leq i\leq n\right) $. Then $P$ is a matrix of full rank, and 
$PP^{T}=\Xi _{n}\left( \sigma \right) $.
\end{proof}

\subsection{Operator norm estimate of $E\left( \Psi _{t}\right) $}

We start with the following elementary linear algebra lemmas.

\begin{lemma}
\label{matrix-norm} Let $A$ be a $m\times n$ matrix. Regarding $A$ as a
linear map from $\mathbb{R}^{n}$ to $\mathbb{R}^{m}$, the operator norm $%
\left\Vert A\right\Vert $ of $A$, defined as 
\begin{equation*}
\left\Vert A\right\Vert =\sup_{v\in \mathbb{R}^{n}\text{, }\left\vert
v\right\vert =1}\frac{\left\vert Av\right\vert }{\left\vert v\right\vert },
\end{equation*}%
is less or equal to $\sqrt{n}$ times the length of its longest column
vector. If the column vectors are orthogonal to each other, then $\left\Vert
A\right\Vert $ is equal to the length of the longest column vector.
\end{lemma}

\begin{lemma}
\label{matrix-perturb-inverse}Let $A_{i}\left( t\right) $ be $m_{i}\times
m_{i}$ symmetric, invertible matrices with operator norm $\left\Vert \left(
A_{i}\left( t\right) \right) ^{-1}\right\Vert \leq \rho _{0}$ for $i=1,2$
and $t\in (0,t_{0}]$, and $b\left( t\right) $ be a $m_{2}\times m_{1}$
matrix with $\left\Vert b\left( t\right) \right\Vert \rightarrow 0$ as $%
t\rightarrow 0_{+}$. Then for sufficiently small $t>0$, 
\begin{equation*}
\left[ 
\begin{array}{cc}
A_{1}\left( t\right)  & b^{T}\left( t\right)  \\ 
b\left( t\right)  & A_{2}\left( t\right) 
\end{array}%
\right] ^{-1}=\left[ 
\begin{array}{cc}
A_{1}^{-1}\left( t\right)  & c^{T}\left( t\right)  \\ 
c\left( t\right)  & A_{2}^{-1}\left( t\right) 
\end{array}%
\right] \left[ 
\begin{array}{cc}
\left( I_{m_{1}}+b^{T}\left( t\right) c\left( t\right) \right) ^{-1} & 0 \\ 
0 & \left( I_{m_{2}}+b\left( t\right) c^{T}\left( t\right) \right) ^{-1}%
\end{array}%
\right] ,
\end{equation*}%
where $c\left( t\right) $ is the $m_{2}\times m_{1}$ matrix given by $%
c\left( t\right) =A_{2}^{-1}\left( t\right) b\left( t\right)
A_{1}^{-1}\left( t\right) $. Especially 
\begin{equation*}
\left\Vert c\left( t\right) \right\Vert \leq \left\Vert \left( A_{2}\left(
t\right) \right) ^{-1}\right\Vert \left\Vert b\left( t\right) \right\Vert
\left\Vert \left( A_{1}\left( t\right) \right) ^{-1}\right\Vert \text{.}
\end{equation*}
\end{lemma}

From now we consider the \emph{normalized heat kernel embedding }$\Psi _{t}=%
\sqrt{2}\left( 4\pi \right) ^{n/4}t^{\frac{n+2}{4}}\Phi _{t}$. Theorem \ref%
{Uni-indep} still holds if we replace $\Phi _{t}$ by $\Psi _{t}$, for they
only differ by a scaling factor $\sqrt{2}\left( 4\pi \right) ^{n/4}t^{\frac{%
n+2}{4}}$.

\begin{corollary}
The matrix $P\left( \Psi _{t}\right) \left( x\right) $ has a right inverse $%
E\left( \Psi _{t}\right) \left( x\right) $ with uniform operator norm bound $%
C$ for all $q\geq t^{-\frac{n}{2}-\rho }$ and all $x\in M$ as $t\rightarrow
0_{+}$.
\end{corollary}

\begin{proof}
Since $\Psi _{t}=\sqrt{2}\left( 4\pi \right) ^{n/4}t^{\frac{n+2}{4}}\Phi _{t}
$, by Proposition \ref{heat-asymp}, as $t\rightarrow 0_{+}$ we have in the $%
C^{r}$-sense $\left( \text{for any }r\geq 0\right) $%
\begin{equation}
P\left( \Psi _{t}\right) P^{T}\left( \Psi _{t}\right) \left( x\right) =\left[
\begin{array}{cc}
I_{n}+O\left( t\right)  & O\left( t\right)  \\ 
O\left( t\right)  & \frac{1}{2t}\cdot \left( \left[ 
\begin{array}{cc}
I_{\frac{n\left( n-1\right) }{2}} & 0 \\ 
0 & \Xi _{n}\left( \frac{1}{3}\right) 
\end{array}%
\right] +O\left( t\right) \right) 
\end{array}%
\right] ,  \label{PPT}
\end{equation}%
where $I_{n}$ corresponds to $\left\langle \nabla _{i}\Psi _{t},\nabla
_{j}\Psi _{t}\right\rangle $, $I_{\frac{n\left( n-1\right) }{2}}$
corresponds to $\left\langle \nabla _{i}\nabla _{j}\Psi _{t},\nabla
_{k}\nabla _{l}\Psi _{t}\right\rangle $ for $i\neq j$ and $k\neq l$, and $%
\Xi _{n}\left( \frac{1}{3}\right) $ (defined in Corollary \ref{Theta-matrix}%
) corresponds to $\left\langle \nabla _{i}\nabla _{i}\Psi _{t},\nabla
_{k}\nabla _{k}\Psi _{t}\right\rangle $ for $1\leq i,j,k,l\leq n$. By
Proposition \ref{isom-truncation-high-jet}, $\left( \ref{PPT}\right) $ still
holds when we truncate $\Psi _{t}$ from $\ell ^{2}$ to $\mathbb{R}^{q\left(
t\right) }$ with $q\left( t\right) \geq $ $t^{-\frac{n}{2}-\rho }$.
Therefore as $t\rightarrow 0_{+}$, by Lemma \ref{matrix-perturb-inverse} we
have 
\begin{equation}
\left[ P\left( \Psi _{t}\right) P^{T}\left( \Psi _{t}\right) \right]
^{-1}\left( x\right) =\left[ 
\begin{array}{cc}
I_{n}+O\left( t\right)  & O\left( t\right)  \\ 
O\left( t\right)  & 2t\cdot \left( \left[ 
\begin{array}{cc}
I_{\frac{n\left( n-1\right) }{2}} & 0 \\ 
0 & \left( \Xi _{n}\left( \frac{1}{3}\right) \right) ^{-1}%
\end{array}%
\right] +O\left( t\right) \right) 
\end{array}%
\right]   \label{PPT-inverse}
\end{equation}%
in the $C^{r}$-sense. By Lemma \ref{matrix-norm}, for the right inverse 
\begin{equation}
E\left( \Psi _{t}\right) =P^{T}\left( \Psi _{t}\right) \left[ P\left( \Psi
_{t}\right) P^{T}\left( \Psi _{t}\right) \right] ^{-1},  \label{Epsi_t}
\end{equation}%
its operator norm is controlled by the length of its longest column vector.
From Proposition \ref{heat-asymp}, $\left\vert \nabla _{i}\Psi
_{t}\right\vert \rightarrow 1$ and $\left\vert \nabla _{j}\nabla _{k}\Psi
_{t}\right\vert \rightarrow \frac{1}{\sqrt{2t}}$ in the $C^{r}$-sense as $%
t\rightarrow 0_{+}$, so plugging this and $\left( \ref{PPT-inverse}\right) $
into $\left( \ref{Epsi_t}\right) $ we have 
\begin{equation*}
\left\Vert E\left( \Psi _{t}\right) \left( x\right) \right\Vert \leq
C\sup_{1\leq i,j,k\leq n}\left( \left\Vert \nabla _{i}\Psi _{t}\left(
x\right) \right\Vert +t\left\Vert \nabla _{j}\nabla _{k}\Psi _{t}\left(
x\right) \right\Vert \right) \leq C.
\end{equation*}
\end{proof}

\begin{proposition}
\label{point-operator-norm}For any multi-index $\overrightarrow{\alpha }$
with $\left\vert \overrightarrow{\alpha }\right\vert =k$ and $0<\alpha <1$,
for any $x\in M\,$, the operator norms of the linear maps $D^{%
\overrightarrow{\alpha }}E\left( \Psi _{t}\right) \left( x\right) $ and $%
\left[ D^{\overrightarrow{\alpha }}E\left( \Psi _{t}\right) \right] _{\alpha
,M}\left( x\right) :\mathbb{R}^{\frac{n\left( n+3\right) }{2}}\rightarrow 
\mathbb{R}^{q\left( t\right) }$ satisfy%
\begin{eqnarray}
\left\Vert D^{\overrightarrow{\alpha }}E\left( \Psi _{t}\right) \left(
x\right) \right\Vert  &\leq &Ct^{-\frac{k}{2}},  \notag \\
\left\Vert \left[ D^{\overrightarrow{\alpha }}E\left( \Psi _{t}\right) %
\right] _{\alpha ,M}\left( x\right) \right\Vert  &\leq &Ct^{-\frac{k+\alpha 
}{2}}  \label{Eu-psi_t-Schauder}
\end{eqnarray}%
respectively for all $q\left( t\right) \geq t^{-\frac{n}{2}-\rho }$ as $%
t\rightarrow 0_{+}$.
\end{proposition}

\begin{proof}
In each chart $U$ of $M$ where we use the orthonormal frame field $\left\{
V_{i}\right\} _{1\leq i\leq n}$ to trivialize $P\left( \Psi _{t}\right) $,
the $O\left( t\right) $'s in $\left( \ref{PPT-inverse}\right) $ are smooth
functions on $U$ with the $C^{r}$-norm of order $O\left( t\right) $. So for
any multi-index $\overrightarrow{\gamma }$ and real number $\gamma \in
\lbrack 0,1)$ with $\left\vert \overrightarrow{\gamma }\right\vert +\gamma >0
$, by $\left( \ref{PPT-inverse}\right) $ we have 
\begin{equation}
\left[ D^{\overrightarrow{\gamma }}\left[ P\left( \Psi _{t}\right)
P^{T}\left( \Psi _{t}\right) \right] ^{-1}\right] _{\gamma }\left( x\right)
=O\left( t\right)   \label{small-off-diagonal-deri}
\end{equation}%
in the $C^{r}$-sense for any $r\geq 0$. Therefore for any multi-index $%
\overrightarrow{\alpha }$ with $\left\vert \overrightarrow{\alpha }%
\right\vert =k$, applying $D^{\overrightarrow{\alpha }}$ to $\left( \ref%
{Epsi_t}\right) $, using the Leibniz rule, and noticing Lemma \ref%
{matrix-norm} and $\left( \ref{small-off-diagonal-deri}\right) $ we have 
\begin{eqnarray*}
\left\Vert D^{\overrightarrow{\alpha }}E\left( \Psi _{t}\right) \text{ }%
\left( x\right) \right\Vert  &\leq &C\Sigma _{\overrightarrow{\beta }\cup 
\overrightarrow{\gamma }=\overrightarrow{\alpha }}\left\Vert D^{%
\overrightarrow{\beta }}P^{T}\left( \Psi _{t}\right) \left( x\right) \cdot
D^{\overrightarrow{\gamma }}\left[ P\left( \Psi _{t}\right) P^{T}\left( \Psi
_{t}\right) \right] ^{-1}\left( x\right) \right\Vert  \\
&\leq &C\left( \left\vert \nabla ^{k+1}\Psi _{t}\left( x\right) \right\vert
+\left\vert \nabla ^{k+2}\Psi _{t}\left( x\right) \right\vert \cdot O\left(
t\right) \right) \leq Ct^{-\frac{k}{2}}, \\
\left\Vert \left[ D^{\overrightarrow{\alpha }}E\left( \Psi _{t}\right) %
\right] _{\alpha ,M}\left( x\right) \right\Vert  &\leq &C\left( \left\Vert %
\left[ D^{\overrightarrow{\alpha }}P^{T}\left( \Psi _{t}\right) \left(
x\right) \right] _{\alpha }\cdot \left[ P\left( \Psi _{t}\right) P^{T}\left(
\Psi _{t}\right) \right] ^{-1}\left( x\right) \right\Vert \right.  \\
&&\left. +\left\Vert D^{\overrightarrow{\alpha }}P^{T}\left( \Psi
_{t}\right) \left( x\right) \cdot \left[ \left[ P\left( \Psi _{t}\right)
P^{T}\left( \Psi _{t}\right) \right] ^{-1}\left( x\right) \right] _{\alpha
}\right\Vert \right)  \\
&\leq &C\left( t^{-\frac{k}{2}-\frac{\alpha }{2}}+t^{-\frac{k}{2}}\cdot
O\left( t\right) \right) \leq Ct^{-\frac{k+\alpha }{2}},
\end{eqnarray*}%
where in both inequalities we have used Proposition \ref{high-derivative}
and Proposition \ref{Schauder-derivative-est} for the derivative estimates
of $\Psi _{t}\left( x\right) $.\ 
\end{proof}

\begin{corollary}
\label{C2a-operator-norm}For $q\geq Ct^{-\frac{n}{2}-\rho }$, $\left\Vert
E\left( \Psi _{t}\right) \right\Vert _{C^{k,\alpha }\left( M\right) }$ is of
order $t^{-\frac{k+\alpha }{2}}$, so is the \emph{operator norm} $\left\Vert
E\left( \Psi _{t}\right) \right\Vert _{\text{op}}$ of $E\left( \Psi
_{t}\right) :C^{k,\alpha }\left( M,T^{\ast }M\right) \times C^{k,\alpha
}\left( M,\mathrm{Sym}^{\otimes 2}\left( T^{\ast }M\right) \right)
\rightarrow C^{k,\alpha }\left( M,\mathbb{R}^{q}\right) $, i.e.%
\begin{equation}
\left\Vert E\left( \Psi _{t}\right) \right\Vert _{C^{k,\alpha }\left(
M\right) },\left\Vert E\left( \Psi _{t}\right) \right\Vert _{\mathrm{op}%
}\leq Ct^{-\frac{k+\alpha }{2}}  \label{C_E}
\end{equation}%
for a constant $C>0$.
\end{corollary}

\begin{proof}
Taking the supremum for $x\in M$ in the inequalities in $\left( \ref%
{Eu-psi_t-Schauder}\right) $, we obtain%
\begin{equation}
\left\Vert E\left( \Psi _{t}\right) \right\Vert _{C^{k,\alpha }\left(
M\right) }\leq Ct^{-\frac{k+\alpha }{2}}\text{.}  \label{Eu-psi_t_Schauder}
\end{equation}%
Now we estimate the operator norm $\left\Vert E\left( \Psi _{t}\right)
\right\Vert _{\mathrm{op}}$. For any section $\varphi \in C^{k,\alpha
}\left( M,T^{\ast }M\right) \times C^{k,\alpha }\left( M,\mathrm{Sym}%
^{\otimes 2}\left( T^{\ast }M\right) \right) $, using Proposition \ref%
{point-operator-norm}, by the Leibniz rule we have%
\begin{equation*}
\left\vert \left[ D^{\overrightarrow{\alpha }}\left( E\left( \Psi
_{t}\right) \left( x\right) \varphi \left( x\right) \right) \right] _{\alpha
,M}\right\vert \leq Ct^{-\frac{k+\alpha }{2}}\left\Vert \varphi \right\Vert
_{C^{k,\alpha }\left( M\right) }
\end{equation*}%
for any multi-index $\overrightarrow{\alpha }$ with $\left\vert 
\overrightarrow{\alpha }\right\vert =k$. (In a local trivialization of $TM$, 
$\varphi :M\rightarrow \mathbb{R}^{\frac{n\left( n+3\right) }{2}}$. The $%
C^{k,\alpha }$-norm for vector-valued functions is given in Section \ref%
{Schauder-norm-Rq}). Taking the supremum for $x\in M$ in the above
inequalities, we have%
\begin{equation*}
\left\Vert E\left( u\right) \varphi \right\Vert _{C^{k,\alpha }\left( M,%
\mathbb{R}^{q}\right) }\leq Ct^{-\frac{k+\alpha }{2}}\left\Vert \varphi
\right\Vert _{C^{k,\alpha }\left( M,\mathbb{R}^{q}\right) },
\end{equation*}%
so the \emph{operator norm }$\left\Vert E\left( \Psi _{t}\right) \right\Vert
_{\mathrm{op}}$ is of order $Ct^{-\frac{k+\alpha }{2}}$. Note this operator
norm agrees with the $C^{k,\alpha }$-H\"{o}lder norm of $E\left( u\right) $
by $\left( \ref{Eu-psi_t_Schauder}\right) $.
\end{proof}

\begin{definition}
\label{CE}\textbf{\ (}The constant\textbf{\ }$C_{E}$\textbf{): }Due to the
importance of the operator norm of $E\left( \Psi _{t}\right) $, we denote
the maximum of the constants $C$ appeared in the coefficients of the above
estimates of $\left\Vert E\left( \Psi _{t}\right) \left( x\right)
\right\Vert $, $\left\Vert E\left( \Psi _{t}\right) \right\Vert $, $%
\left\Vert E\left( \Psi _{t}\right) \right\Vert _{C^{k,\alpha }\left( M,%
\mathbb{R}^{q}\right) }$, $\left\Vert D^{\overrightarrow{\alpha }}\Psi
_{t}\right\Vert _{C^{k,\alpha }\left( M,\mathbb{R}^{q}\right) }$ in
Proposition \ref{Schauder-derivative-est}, $O\left( t\right) $ in $\left( %
\ref{PPT-inverse}\right) $ and $2\left\Vert \left( \Xi \left( \frac{1}{3}%
\right) \right) ^{-1}\right\Vert $by $C_{E}$, where \textquotedblleft $E$%
\textquotedblright\ indicates $E\left( \Psi _{t}\right) $. 
\end{definition}

\section{Uniform quadratic estimate of $Q\left( u\right) \label%
{Qudratic-estimate}$}

For any given map $u\in C^{k,\alpha }\left( M,\mathbb{R}^{q}\right) $, the
quadratic estimate of $Q\left( u\right) $ was established in \cite{G1} Lemma
4. In this section we show the constant in the quadratic estimate is \emph{%
uniform} for all $\mathbb{R}^{q}$. This is essentially due to Lemma \ref%
{product-inequ}, where the constant $C\left( k,\alpha ,M\right) $ is uniform
for all $q$.

\begin{proposition}
\label{Quadratic}For \thinspace any $v\in C^{k,\alpha }\left( M,\mathbb{R}%
^{q}\right) $, we have%
\begin{eqnarray*}
\left\Vert Q_{i}\left( v,v\right) \right\Vert _{C^{k,\alpha }\left( M,%
\mathbb{R}^{q}\right) } &\leq &\Gamma \left( \Lambda _{0},k,\alpha
,\left\Vert R\right\Vert _{C^{1}}\right) \left\Vert v\right\Vert
_{C^{k,\alpha }\left( M,\mathbb{R}^{q}\right) }^{2}\text{, \ } \\
\left\Vert Q_{ij}\left( v,v\right) \right\Vert _{C^{k,\alpha }\left( M,%
\mathbb{R}^{q}\right) } &\leq &\Gamma \left( \Lambda _{0},k,\alpha
,\left\Vert R\right\Vert _{C^{1}}\right) \left\Vert v\right\Vert
_{C^{k,\alpha }\left( M,\mathbb{R}^{q}\right) }^{2}, \\
\left\Vert Q\left( \Psi _{t}\right) \left( v,v\right) \right\Vert
_{C^{k,\alpha }\left( M,\mathbb{R}^{q}\right) } &\leq &C_{E}\Gamma \left(
\Lambda _{0},k,\alpha ,\left\Vert R\right\Vert _{C^{1}}\right) t^{-\frac{k}{2%
}-\frac{\alpha }{2}}\left\Vert v\right\Vert _{C^{k,\alpha }\left( M,\mathbb{R%
}^{q}\right) }^{2}\text{,}
\end{eqnarray*}%
where the constant $\Gamma \left( \Lambda _{0},k,\alpha ,\left\Vert
R\right\Vert _{C^{1}}\right) $ is uniform for all $q$, where $\left\Vert
R\right\Vert _{C^{1}}$ is the $C^{1}$-norm of the Riemannian curvature
tensor $R$ on $M$. (The constants $\sigma \left( \Lambda _{0},\alpha
,M\right) $, $C\left( k,\alpha ,M\right) $ and $C_{E}$ are in $\left( \ref%
{spectral-constant}\right) $, Lemma \ref{product-inequ} and Definition \ref%
{CE} respectively).
\end{proposition}

\begin{proof}
For brevity we write $C^{k,\alpha }\left( M,\mathbb{R}^{q}\right) $ as $%
C^{k,\alpha }\left( M\right) $. Let $\cdot $ be the standard inner product
in $\mathbb{R}^{q}$. Recall that $Q_{i}\left( v,v\right) =\left( \left(
\Delta _{\left( 1\right) }-\Lambda _{0}\right) ^{-1}\left( N\left( v\right)
\right) \right) _{i}$, and $Q_{ij}\left( v,v\right) =\left( \left( \Delta
_{\left( 2\right) }-\Lambda _{0}\right) ^{-1}\left( M\left( v\right) \right)
\right) _{ij}$, where 
\begin{eqnarray*}
N_{i}\left( v\right)  &=&-\Delta v\cdot \nabla _{i}v, \\
L_{ij}\left( v\right)  &=&2\nabla ^{l}\nabla _{i}v\cdot \nabla _{l}\nabla
_{j}v-2\Delta v\cdot \nabla _{i}\nabla _{j}v-2R_{i.j.}^{kl}\nabla _{k}v\cdot
\nabla _{l}v-\Lambda _{0}\nabla _{i}v\cdot \nabla _{j}v, \\
M_{ij}\left( v\right)  &=&\frac{1}{2}L_{ij}\left( v\right) +\left( \nabla
_{i}R_{j.}^{l}+\nabla _{j}R_{i.}^{l}-\nabla ^{l}R_{ij}\right) \left( \left(
\Delta _{\left( 1\right) }-\Lambda _{0}\right) ^{-1}N\left( v\right) \right)
_{l}.
\end{eqnarray*}%
By the definition of the operator norm $\left\Vert \left( \Delta _{\left(
r\right) }-\Lambda _{0}\right) ^{-1}\right\Vert _{\text{op}}$ and $\left( %
\ref{spectral-constant}\right) $, we have%
\begin{eqnarray}
&&\left\Vert Q_{ij}\left( v,v\right) \right\Vert _{C^{k,\alpha }\left(
M\right) }  \label{Qij} \\
&\leq &\sigma \left( \Lambda _{0},\alpha ,M\right) \left( \frac{1}{2}%
\left\Vert L_{ij}\left( v\right) \right\Vert _{C^{k-2,\alpha }\left(
M\right) }+\left\Vert \nabla R\right\Vert _{C^{0}\left( M\right) }\sigma
\left( \Lambda _{0},\alpha ,M\right) \left\Vert N\left( v\right) \right\Vert
_{C^{k-2,\alpha }\left( M\right) }\right) .  \notag
\end{eqnarray}%
For $L_{ij}\left( v\right) $, by $\left( \ref{Cka}\right) $ we have%
\begin{eqnarray*}
&&\left\Vert L_{ij}\left( v\right) \right\Vert _{C^{k,\alpha }\left(
M\right) } \\
&\leq &2C\left( k,\alpha ,M\right) \left( \left\Vert \nabla ^{l}\nabla
_{i}v\right\Vert _{C^{k-2,\alpha }\left( M\right) }\left\Vert \nabla
_{l}\nabla _{j}v\right\Vert _{C^{k-2,\alpha }\left( M\right) }+\left\Vert
\Delta v\right\Vert _{C^{k-2,\alpha }\left( M\right) }\left\Vert \nabla
_{i}\nabla _{j}v\right\Vert _{C^{k-2,\alpha }\left( M\right) }\right.  \\
&&\left. +\left\Vert R\right\Vert _{C^{0}\left( M\right) }\left\Vert \nabla
_{k}v\right\Vert _{C^{k-2,\alpha }\left( M\right) }\left\Vert \nabla
_{l}v\right\Vert _{C^{k-2,\alpha }\left( M\right) }+\left\vert \frac{\Lambda
_{0}}{2}\right\vert \left\Vert \nabla _{i}v\right\Vert _{C^{k-2,\alpha
}\left( M\right) }\left\Vert \nabla _{j}v\right\Vert _{C^{k-2,\alpha }\left(
M\right) }\right)  \\
&\leq &2C\left( k,\alpha ,M\right) \left( \left\Vert v\right\Vert
_{C^{k,\alpha }\left( M\right) }^{2}+\left\Vert v\right\Vert _{C^{k,\alpha
}\left( M\right) }^{2}+\left\Vert R\right\Vert _{C^{0}\left( M\right)
}\left\Vert v\right\Vert _{C^{k-1,\alpha }\left( M\right) }^{2}+\left\vert 
\frac{\Lambda _{0}}{2}\right\vert \left\Vert \nabla _{i}v\right\Vert
_{C^{k-1,\alpha }\left( M\right) }^{2}\right)  \\
&\leq &2C\left( k,\alpha ,M\right) \left( 2+\left\Vert R\right\Vert
_{C^{0}\left( M\right) }+\left\vert \frac{\Lambda _{0}}{2}\right\vert
\right) \left\Vert v\right\Vert _{C^{k,\alpha }\left( M\right) }^{2}.
\end{eqnarray*}%
Similarly%
\begin{equation*}
\left\Vert N_{i}\left( v\right) \right\Vert _{C^{k-2,\alpha }\left( M\right)
}\leq \left\Vert \Delta v\right\Vert _{C^{k-2,\alpha }\left( M\right)
}\left\Vert \nabla _{i}v\right\Vert _{C^{k-2,\alpha }\left( M\right) }\leq
\left\Vert v\right\Vert _{C^{k,\alpha }\left( M\right) }^{2}.
\end{equation*}%
Putting these into $\left( \ref{Qij}\right) $ we have%
\begin{eqnarray}
\left\Vert Q_{ij}\left( v,v\right) \right\Vert _{C^{k,\alpha }\left(
M\right) } &\leq &\sigma \left( \Lambda _{0},\alpha ,M\right) C\left(
k,\alpha ,M\right) \left( 2+\left\Vert R\right\Vert _{C^{0}\left( M\right)
}+\left\vert \frac{\Lambda _{0}}{2}\right\vert \right) \left\Vert
v\right\Vert _{C^{k,\alpha }\left( M\right) }^{2}  \notag \\
&&+\sigma ^{2}\left( \Lambda _{0},\alpha ,M\right) \left\Vert \nabla
R\right\Vert _{C^{0}\left( M\right) }\left\Vert v\right\Vert _{C^{k,\alpha
}\left( M\right) }^{2}  \notag \\
&=&\Gamma \left( \Lambda _{0},k,\alpha ,\left\Vert R\right\Vert
_{C^{1}}\right) \left\Vert v\right\Vert _{C^{k,\alpha }\left( M\right) }^{2}%
\text{,}  \label{quadratic-Q-bd}
\end{eqnarray}%
where the constant%
\begin{eqnarray}
&&\Gamma \left( \Lambda _{0},k,\alpha ,\left\Vert R\right\Vert
_{C^{1}}\right)   \label{Gamma} \\
&:&=\sigma \left( \Lambda _{0},\alpha ,M\right) C\left( k,\alpha ,M\right)
\left( 2+\left\Vert R\right\Vert _{C^{0}\left( M\right) }+\left\vert \frac{%
\Lambda _{0}}{2}\right\vert \right) +\sigma ^{2}\left( \Lambda _{0},\alpha
,M\right) \left\Vert \nabla R\right\Vert _{C^{0}\left( M\right) }\text{.} 
\notag
\end{eqnarray}%
Similarly%
\begin{equation*}
\left\Vert Q_{i}\left( v,v\right) \right\Vert _{C^{k,\alpha }\left( M\right)
}\leq \sigma \left( \Lambda _{0},\alpha ,M\right) C\left( k,\alpha ,M\right)
\left\Vert v\right\Vert _{C^{k,\alpha }\left( M\right) }^{2}\leq \Gamma
\left( \Lambda _{0},k,\alpha ,\left\Vert R\right\Vert _{C^{1}}\right)
\left\Vert v\right\Vert _{C^{k,\alpha }\left( M\right) }^{2}\text{.}
\end{equation*}%
Finally, since 
\begin{equation*}
Q\left( u\right) \left( v,v\right) =E\left( u\right) \left( \left[
Q_{i}\left( u\right) \left( v,v\right) \right] ,\left[ Q_{ij}\left( u\right)
\left( v,v\right) \right] \right) 
\end{equation*}%
and the \emph{operator norms} of $E\left( \Psi _{t}\right) :C^{k,\alpha
}\left( M,T^{\ast }M\right) \times C^{k,\alpha }\left( M,\mathrm{Sym}%
^{\otimes 2}\left( T^{\ast }M\right) \right) \rightarrow C^{k,\alpha }\left(
M,\mathbb{R}^{q}\right) $ is of order $C_{E}t^{-\frac{k}{2}-\frac{\alpha }{2}%
}$ by Corollary \ref{C2a-operator-norm}, we have%
\begin{eqnarray}
\left\Vert Q\left( \Psi _{t}\right) \left( v,v\right) \right\Vert
_{C^{k,\alpha }\left( M\right) } &\leq &\Gamma \left( \Lambda _{0},k,\alpha
,\left\Vert R\right\Vert _{C^{1}}\right) \left\Vert E\left( \Psi _{t}\right)
\right\Vert _{C^{k,\alpha }\left( M\right) }\left\Vert v\right\Vert
_{C^{k,\alpha }\left( M\right) }^{2}  \label{quadratic-est} \\
&\leq &C_{E}\Gamma \left( \Lambda _{0},k,\alpha ,\left\Vert R\right\Vert
_{C^{1}}\right) t^{-\frac{k+\alpha }{2}}\left\Vert v\right\Vert
_{C^{k,\alpha }\left( M\right) }^{2},  \notag
\end{eqnarray}%
where the constant $C_{E}\Gamma \left( \Lambda _{0},k,\alpha ,\left\Vert
R\right\Vert _{C^{1}}\right) $ is uniform for all $q$.
\end{proof}

\section{The implicit function theorem: isometric embedding\label{IFT}}

\bigskip In previous sections we have considered the $\frac{n\left(
n+3\right) }{2}\times \infty $ matrix $P\left( \Psi _{t}\right) $ and its
right inverse $E\left( \Psi _{t}\right) $. If we truncate $\ell ^{2}$ to $%
\mathbb{R}^{q\left( t\right) }$ with $q\left( t\right) \geq Ct^{-\frac{n}{2}%
-\rho }$, and consider the \emph{modified heat kernel embedding map} $\tilde{%
\Psi}_{t}:M\rightarrow $ $\mathbb{R}^{q\left( t\right) }$, then $E\left( 
\tilde{\Psi}_{t}\right) $ is a $q\left( t\right) \times \frac{n\left(
n+3\right) }{2}$ matrix. For each fixed $t$, the modified $\tilde{\Psi}%
_{t}=\Psi _{t,g\left( t\right) }$ is the heat kernel embedding map in \cite%
{BBG} for the \emph{modified metric} $g_{t}$. The modified metrics $\left\{
g_{s}\right\} _{0\leq s\leq t_{0}}$ is a compact family, and depend on $s$
smoothly. By Proposition \ref{isom-truncation-high-jet} and Proposition \ref%
{Schauder-derivative-est}, $E\left( \tilde{\Psi}_{t}\right) $ still has the
operator bounds as in Proposition \ref{point-operator-norm} and Corollary %
\ref{C2a-operator-norm} for $E\left( \Psi _{t}\right) $. This is because
from our construction of $E\left( \Psi _{t,g\left( s\right) }\right) $ $%
\left( \ref{Epsi_t}\right) $, $\left\Vert E\left( \Psi _{t,g\left( s\right)
}\right) \right\Vert _{C^{k,\alpha }\left( M\right) }$ is determined by the
(derivatives of) inner products of the row vectors $\partial _{i}\Psi
_{t,g\left( s\right) }$ and $\partial _{i}\partial _{j}\Psi _{t,g\left(
s\right) }$ ($1\leq i\leq j\leq n$) for the parameter $s$, which is in a 
\emph{compact} interval $\left[ 0,t_{0}\right] $.

Now we are ready to give the proof of Theorem \ref{isom-emb}. We divide the
proof into two propositions: isometric immersion and one-to-one map.

\begin{proposition}
\label{isom-immersion}(Isometric immersion) Under the conditions of Theorem %
\ref{isom-emb}, there exists $t_{0}>0$ depending on $\left( g,\rho ,\alpha
\right) $, such that for the integer $q=q\left( t\right) \geq t^{-\frac{n}{2}%
-\rho }$ and $0<t\leq t_{0}$, the modified heat kernel embedding $\tilde{\Psi%
}_{t}$ can be truncated to 
\begin{equation*}
\tilde{\Psi}_{t}:M\rightarrow \mathbb{R}^{q}\subset \ell ^{2}
\end{equation*}%
and can be perturbed to an isometric embedding $I_{t}:M\rightarrow \mathbb{R}%
^{q}$, with the perturbation of $\tilde{\Psi}_{t}$ of order $O\left( t^{%
\frac{k+1}{2}-\frac{\alpha }{2}}\right) $ in the $C^{k,\alpha }$-norm.
\end{proposition}

\begin{proof}
Given the truncated heat kernel embedding $u=\tilde{\Psi}_{t}:M\rightarrow 
\mathbb{R}^{q\left( t\right) }$ with $q=q\left( t\right) \geq t^{\frac{n}{2}%
-\rho }$, and the error $f:=\left( \tilde{\Psi}_{t}\right) ^{\ast }g_{can}-g$
to the isometric embedding, we consider the nonlinear functional 
\begin{eqnarray}
F &:&C^{k,\alpha }\left( M,\mathbb{R}^{q}\right) \rightarrow C^{k,\alpha
}\left( M,\mathbb{R}^{q}\right) ,  \label{F} \\
F\left( v\right)  &=&v-E\left( \tilde{\Psi}_{t}\right) \left( 0,f\right)
+E\left( \tilde{\Psi}_{t}\right) \left( \left[ Q_{i}\left( \tilde{\Psi}%
_{t}\right) \left( v,v\right) \right] ,\left[ Q_{jk}\left( \tilde{\Psi}%
_{t}\right) \left( v,v\right) \right] \right) .  \notag
\end{eqnarray}%
We stress that this iteration is \emph{coordinate free} and is defined on
the \emph{whole }$M$, as it is the coordinate expression of the iteration of
tensors (see equations (12)\symbol{126}(21) in \cite{G1}). We want to find
the zeros of $F$. By the general implicit function theorem (e.g. Proposition
A.3.4. in \cite{MS}), the operator norm estimate in Corollary \ref%
{C2a-operator-norm}, and the uniform quadratic estimates in Proposition \ref%
{Quadratic}, it is enough to verify that%
\begin{equation*}
\left\Vert E\left( \tilde{\Psi}_{t}\right) \right\Vert _{C^{k,\alpha }\left(
M\right) }\left\Vert E\left( \tilde{\Psi}_{t}\right) \left( 0,f\right)
\right\Vert _{C^{k,\alpha }\left( M,\mathbb{R}^{q}\right) }\rightarrow 0
\end{equation*}%
as $t\rightarrow 0_{+}$. By Corollary \ref{C2a-operator-norm} we have%
\begin{equation*}
\left\Vert E\left( \tilde{\Psi}_{t}\right) \right\Vert _{C^{k,\alpha }\left(
M\right) }\leq C_{E}t^{-\frac{k}{2}-\frac{\alpha }{2}}.
\end{equation*}%
By Theorem \ref{isom-truncation} we have $f=\left( \tilde{\Psi}_{t}\right)
^{\ast }g_{can}-g=O\left( t^{l}\right) $ in the $C^{k+1}$ norm, so for small 
$t$, 
\begin{equation}
\left\Vert f\right\Vert _{C^{k,\alpha }\left( M,\mathrm{Sym}^{\otimes
2}\left( T^{\ast }M\right) \right) }\leq Gt^{l}  \label{h-error}
\end{equation}%
for the constant 
\begin{equation}
G:=C\left( g,l,k+1\right)   \label{G-constant}
\end{equation}%
in Proposition \ref{Higher-order-error} (When $k=l=2$, we have partial
estimate of $G$ by curvature terms in Section \ref{remainder}). By our
construction in $\left( \ref{Epsi_t}\right) $ and $\left( \ref{PPT-inverse}%
\right) $, we have

\begin{equation*}
E\left( \tilde{\Psi}_{t}\right) \left( 0,f\right) =P^{T}\left( \tilde{\Psi}%
_{t}\right) \left[ 
\begin{array}{c}
O\left( t\right) \cdot f \\ 
2t\cdot \left( \left[ 
\begin{array}{cc}
I_{\frac{n\left( n-1\right) }{2}} & 0 \\ 
0 & \left( \Xi _{n}\left( \frac{1}{3}\right) \right) ^{-1}%
\end{array}%
\right] +O\left( t\right) \right) \cdot f%
\end{array}%
\right] .
\end{equation*}%
When $t$ is small, $\left\vert \nabla _{i}\tilde{\Psi}_{t}\right\vert
<<\left\vert \nabla _{j}\nabla _{k}\tilde{\Psi}_{t}\right\vert $, so we have 
\begin{eqnarray}
\left\Vert E\left( \tilde{\Psi}_{t}\right) \left( 0,f\right) \right\Vert
_{C^{k,\alpha }\left( M,\mathbb{R}^{q}\right) } &=&\left\Vert \left[ \left(
\nabla _{i}\nabla _{k}\tilde{\Psi}_{t}\right) ^{T}\right] _{1\leq i\leq
k\leq n}\cdot O\left( t\right) \cdot f\right\Vert _{C^{k,\alpha }\left( M,%
\mathbb{R}^{q}\right) }  \notag \\
&\leq &C_{E}\left( t^{-\frac{k+1}{2}-\frac{\alpha }{2}}\right) \cdot
C_{E}t\cdot Gt^{l}  \notag \\
&=&C_{E}^{2}Gt^{l+\frac{1}{2}-\frac{k+\alpha }{2}}\text{.}  \label{E-error}
\end{eqnarray}%
where $\left[ \cdot \right] _{1\leq i\leq k\leq n}$ is the notation for a
matrix, and we have used Proposition \ref{point-operator-norm} to estimate $%
\left\Vert \nabla _{i}\nabla _{k}\tilde{\Psi}_{t}\right\Vert _{C^{k,\alpha
}\left( M,\mathbb{R}^{q}\right) }$. Hence 
\begin{eqnarray}
&&\left\Vert E\left( \tilde{\Psi}_{t}\right) \right\Vert _{C^{k,\alpha
}\left( M\right) }\left\Vert E\left( \tilde{\Psi}_{t}\right) \left(
0,f\right) \right\Vert _{C^{k,\alpha }\left( M,\mathbb{R}^{q}\right) } 
\notag \\
&\leq &C_{E}^{3}Gt^{-\frac{k}{2}-\frac{\alpha }{2}}\cdot t^{l+\frac{1}{2}-%
\frac{k+\alpha }{2}}  \notag \\
&=&C_{E}^{3}Gt^{l+\frac{1}{2}-k-\alpha }\rightarrow 0  \label{IFT-small}
\end{eqnarray}%
as $t\rightarrow 0_{+}$, for $l+\frac{1}{2}>k+\alpha $ by our assumption.

The same quadratic estimate still holds for $\tilde{\Psi}_{t}$ for $0<t\leq
t_{0}$ and is \emph{uniform} for all $q\left( t\right) \geq t^{\frac{n}{2}%
-\rho }$, by Corollary \ref{C2a-operator-norm}, Proposition \ref{Quadratic}
and our remark on $\left\Vert E\left( \tilde{\Psi}_{t}\right) \right\Vert
_{C^{k,\alpha }\left( M\right) }$ in the beginning of this subsection, as
follows:%
\begin{eqnarray*}
&&\left\Vert Q\left( \tilde{\Psi}_{t}\right) \left( v,v\right) \right\Vert
_{C^{k,\alpha }\left( M,\mathbb{R}^{q}\right) } \\
&=&\left\Vert E\left( \tilde{\Psi}_{t}\right) \left( \left[ Q_{i}\left( 
\tilde{\Psi}_{t}\right) \left( v,v\right) \right] ,\left[ Q_{ij}\left( 
\tilde{\Psi}_{t}\right) \left( v,v\right) \right] \right) \right\Vert
_{C^{k,\alpha }\left( M,\mathbb{R}^{q}\right) }
\end{eqnarray*}
\begin{eqnarray*}
&\leq &\left\Vert E\left( \tilde{\Psi}_{t}\right) \right\Vert _{C^{k,\alpha
}\left( M\right) }\cdot \Gamma \left( \Lambda _{0},k,\alpha ,\left\Vert
R\right\Vert _{C^{1}}\right) \left\Vert v\right\Vert _{C^{k,\alpha }\left( M,%
\mathbb{R}^{q}\right) }^{2} \\
&\leq &C_{E}\Gamma \left( \Lambda _{0},k,\alpha ,\left\Vert R\right\Vert
_{C^{1}}\right) t^{-\frac{k}{2}-\frac{\alpha }{2}}\left\Vert v\right\Vert
_{C^{k,\alpha }\left( M,\mathbb{R}^{q}\right) }^{2}.
\end{eqnarray*}

By G\"{u}nther's implicit function theorem we obtain a smooth map $I_{t}:$ $%
M\rightarrow \mathbb{R}^{q}$ such that $I_{t}^{\ast }g_{can}=g$. From this
we immediately see $I_{t}$ is an isometric immersion. From the implicit
function theorem we also see the needed perturbation from $\tilde{\Psi}_{t}$
to $I_{t}$ is of order $O\left( t^{l+\frac{1}{2}-\frac{k+\alpha }{2}}\right) 
$ in the $C^{k,\alpha }$-norm (For readers interested in more details about
this, see the Appendix).
\end{proof}

\begin{remark}
$\left( \ref{IFT-small}\right) $ is the place that we need the condition
that $\left\Vert f\right\Vert _{C^{k,\alpha }\left( M,\mathrm{Sym}^{\otimes
2}\left( T^{\ast }M\right) \right) }$ is of order $O\left( t^{l}\right) $
such that $l+\frac{1}{2}>k+\alpha $. Since $k\geq 2$ we must have $l\geq 2$.
This is the reason that we need to modify the $\Psi _{t}$ in \cite{BBG} for
higher order (at least $O\left( t^{2}\right) $) approximation to isometry.
If we can make the remainder terms in $\left( \ref{h-error}\right) $
explicit, then we can give the estimate of the smallness of $t$ in the above
implicit function theorem. See Section \ref{remainder} for partial results
in this direction.
\end{remark}

To show the map $I_{t}:M\rightarrow \mathbb{R}^{q\left( t\right) }$ is
one-to-one for small enough $t>0$, we prove the following

\begin{proposition}
(One-to-one map)\label{truncate-heat-emb-distinguish} Let $\left( M,g\right) 
$ be a compact Riemannian manifold with smooth metric $g$. Then there exists 
$\delta _{0}>0$, such that for $0<t\leq \delta _{0}$, and $q\left( t\right)
\geq Ct^{-\frac{n}{2}-\rho }$, the truncated heat kernel mapping $\Psi
_{t}^{q\left( t\right) }:M\rightarrow \mathbb{R}^{q\left( t\right) }$ can
distinguish any two points on the manifold, i.e. for any $x\neq y$ on $M$, $%
\Psi _{t}^{q\left( t\right) }\left( x\right) \neq \Psi _{t}^{q\left(
t\right) }\left( y\right) $. The same is true for the isometric immersion $%
I_{t}:M\rightarrow \mathbb{R}^{q\left( t\right) }$.
\end{proposition}

\begin{proof}
The proof is adapted from Section 4 of \cite{SZ}. If there is no such $%
\delta _{0}$, then there is a sequence of $t_{k}\rightarrow 0$, and $%
x_{k}\neq y_{k}$ on $M$, such that 
\begin{equation}
\Psi _{t_{k}}^{q\left( t_{k}\right) }\left( x_{k}\right) =\Psi
_{t_{k}}^{q\left( t_{k}\right) }\left( y_{k}\right) \text{, i.e. }\phi
_{j}\left( x_{k}\right) =\phi _{j}\left( y_{k}\right) \text{ for }1\leq
j\leq q\left( t_{k}\right) \text{.}  \label{qt-agree}
\end{equation}%
Therefore%
\begin{equation*}
\Sigma _{j=1}^{q\left( t_{k}\right) }e^{-\lambda _{j}t_{k}}\phi _{j}\left(
x_{k}\right) \phi _{j}\left( y_{k}\right) =\Sigma _{j=1}^{q\left(
t_{k}\right) }e^{-\lambda _{j}t_{k}}\phi _{j}^{2}\left( x_{k}\right) .
\end{equation*}%
By Proposition \ref{isom-truncation-high-jet} and $\left( \ref{phi_t_norm}%
\right) $, letting $k\rightarrow \infty $ we have%
\begin{equation}
\lim_{k\rightarrow \infty }\left( 4\pi t_{k}\right) ^{n/2}H\left(
t_{k},x_{k},y_{k}\right) =\lim_{k\rightarrow \infty }\left( 4\pi
t_{k}\right) ^{n/2}H\left( t_{k},x_{k},x_{k}\right) \rightarrow 1.
\label{heat-agree}
\end{equation}

Let $r_{k}=dist\left( x_{k},y_{k}\right) $. From $\left( \ref{heat-agree}%
\right) $ we see $\lim_{k\rightarrow \infty }r_{k}=0$. Otherwise $r_{k}\geq
c_{0}>0$, by the compactness of $M$ we can assume $x_{k}$ and $y_{k}$
converge to different limits $x_{\infty }\neq y_{\infty }$ on $M$. So in the
left side of $\left( \ref{heat-agree}\right) $ we have $\lim_{k\rightarrow
\infty }H\left( t_{k},x_{k},y_{k}\right) =\lim_{k\rightarrow \infty }H\left(
t_{k},x_{\infty },y_{\infty }\right) =0$, a contradiction. We further claim
that when $k$ is large, 
\begin{equation}
r_{k}:=\text{dist}\left( x_{k},y_{k}\right) \leq At_{k}^{\frac{1}{2}}
\label{closeness-order}
\end{equation}%
for some constant $A>0$. Otherwise, $r_{k}\rightarrow 0$ and $\frac{r_{k}}{%
\sqrt{t_{k}}}\rightarrow +\infty $. By the Minakshisundaram-Pleijel
expansion 
\begin{equation*}
\lim_{k\rightarrow \infty }\left( 4\pi t_{k}\right) ^{n/2}\left\vert H\left(
t_{k},x_{k},y_{k}\right) \right\vert =\lim_{k\rightarrow \infty }e^{-\frac{%
r_{k}^{2}}{4t_{k}}}\left\vert U\left( t_{k},x_{k},y_{k}\right) \right\vert
\leq e^{\lim_{k\rightarrow \infty }\left( -\frac{r_{k}^{2}}{4t_{k}}\right)
}\cdot 2=0,
\end{equation*}%
contradicting with $\left( \ref{heat-agree}\right) $.

By $\left( \ref{closeness-order}\right) $, for large $k\,$, we can write 
\begin{equation*}
y_{k}=\exp _{x_{k}}\left( 2\sqrt{t_{k}}v_{k}\right) \text{, for }0\neq
v_{k}\in T_{x_{k}}M\text{, }\left\vert v_{k}\right\vert =O\left( 1\right) ,
\end{equation*}%
and $y_{k}^{s}=\exp _{x_{k}}\left( sv_{k}\right) $ for $-1\leq s\leq 2$. We
consider the function 
\begin{equation*}
f_{k}\left( s\right) :=\frac{\left( H\left( t_{k},x_{k},y_{k}^{s}\right)
\right) ^{2}}{H\left( t_{k},x_{k},x_{k}\right) H\left(
t_{k},y_{k}^{s},y_{k}^{s}\right) }=\frac{\left\vert \left\langle \Psi
_{t_{k}}\left( x_{k}\right) ,\Psi _{t_{k}}\left( y_{k}^{s}\right)
\right\rangle \right\vert ^{2}}{\left\vert \Psi _{t_{k}}\left( x_{k}\right)
\right\vert ^{2}\left\vert \Psi _{t_{k}}\left( y_{k}^{s}\right) \right\vert
^{2}}\text{, }-1\leq s\leq 2.
\end{equation*}%
By the Cauchy-Schwartz inequality, $0\leq f_{k}\left( s\right) \leq 1$. By
the definition of $f_{k}\left( s\right) $, and our assumption $x_{k}=y_{k}$,
we have $f_{k}\left( 0\right) =f_{k}\left( 1\right) =1$, achieving the
maximum of $f_{k}$ on $\left[ -1,2\right] $. So there exists some $s_{k}\in %
\left[ 0,1\right] $, $f_{k}^{\prime \prime }\left( s_{k}\right) =0$.

In the following we let $z=\left( t,x,x\right) $, $z_{k}=\left(
t_{k},x_{k},x_{k}\right) $ . By $\left( \ref{MP-expansion}\right) $, we have%
\begin{eqnarray}
H\left( t_{k},x_{k},x_{k}\right)  &=&\left( 4\pi t_{k}\right) ^{-n/2}U\left(
z_{k}\right) ,  \label{HxxU} \\
H\left( t_{k},y_{k}^{s},y_{k}^{s}\right)  &=&\left( 4\pi t_{k}\right) ^{-n/2}
\label{HyyU} \\
&&\times \left[ U\left( z_{k}\right) +2s\sqrt{t_{k}}A\left( z_{k}\right)
\left( v_{k}\right) +s^{2}t_{k}E\left( z_{k}\right) \left( v_{k}\right)
+O\left( s^{3}t_{k}^{\frac{3}{2}}\left\vert v_{k}\right\vert ^{3}\right) %
\right]   \notag
\end{eqnarray}%
in the $C^{3}$-norm, where $A\left( t,x,x\right) =\partial _{y}U\left(
t,x,y\right) |_{x=y}$, and $E\left( z\right) \left( v\right) $ is quadratic
in $v$. We also have%
\begin{eqnarray}
H\left( t_{k},x_{k},y_{k}^{s}\right)  &=&\left( 4\pi t_{k}\right)
^{-n/2}e^{-s^{2}\left\vert v_{k}\right\vert ^{2}}U\left(
t_{k},x_{k},y_{k}^{s}\right)   \notag \\
&=&\left( 4\pi t_{k}\right) ^{-n/2}\left[ 1-s^{2}\left\vert v_{k}\right\vert
^{2}+O\left( s^{4}\right) \right]   \label{HxyU} \\
&&\times \left[ U\left( z_{k}\right) +s\sqrt{t_{k}}A\left( z_{k}\right)
\left( v_{k}\right) +s^{2}t_{k}B\left( z_{k}\right) \left( v_{k}\right)
+O\left( s^{3}t_{k}^{\frac{3}{2}}\left\vert v_{k}\right\vert ^{3}\right) %
\right] ,  \notag
\end{eqnarray}%
where $B\left( z\right) \left( v\right) $ is quadratic in $v$. Therefore as $%
k\rightarrow \infty $, in the $C^{3}$-norm we have%
\begin{eqnarray*}
f_{k}\left( s\right)  &=&\left[ 1-2s^{2}\left\vert v_{k}\right\vert
^{2}+O\left( s^{4}\left\vert v_{k}\right\vert ^{4}\right) \right]  \\
&&\times \frac{\left[ U\left( z_{k}\right) +s\sqrt{t_{k}}A\left(
z_{k}\right) \left( v_{k}\right) +s^{2}t_{k}B\left( z_{k}\right) \left(
v_{k}\right) +O\left( s^{3}t_{k}^{\frac{3}{2}}\left\vert v_{k}\right\vert
^{3}\right) \right] ^{2}}{U\left( z_{k}\right) \left[ U\left( z_{k}\right)
+2s\sqrt{t_{k}}A\left( z_{k}\right) \left( v_{k}\right) +s^{2}t_{k}E\left(
z_{k}\right) \left( v_{k}\right) +O\left( s^{3}t_{k}^{\frac{3}{2}}\left\vert
v_{k}\right\vert ^{3}\right) \right] }
\end{eqnarray*}%
\begin{eqnarray*}
&=&\left[ 1-2s^{2}\left\vert v_{k}\right\vert ^{2}+O\left( s^{4}\left\vert
v_{k}\right\vert ^{4}\right) \right] \left[ \frac{1+2s\sqrt{t_{k}}\widetilde{%
A}\left( z_{k}\right) \left( v_{k}\right) +s^{2}t_{k}\widetilde{B}\left(
z_{k}\right) \left( v_{k}\right) +O\left( s^{3}t_{k}^{\frac{3}{2}}\left\vert
v_{k}\right\vert ^{3}\right) }{1+2s\sqrt{t_{k}}\widetilde{A}\left(
z_{k}\right) \left( v_{k}\right) +s^{2}t_{k}\widetilde{E}\left( z_{k}\right)
\left( v_{k}\right) +O\left( s^{3}t_{k}^{\frac{3}{2}}\left\vert
v_{k}\right\vert ^{3}\right) }\right]  \\
&&\text{(Here }\widetilde{A}\text{, }\widetilde{B}\text{ and }\widetilde{E}%
\text{ are functions }A\text{, }B\text{, and }E\text{ divided by }U\left(
z_{k}\right) \text{, respectively)}
\end{eqnarray*}%
\begin{eqnarray*}
&=&\left[ 1-2s^{2}\left\vert v_{k}\right\vert ^{2}+O\left( s^{4}\left\vert
v_{k}\right\vert ^{4}\right) \right] \left[ 1+s^{2}t_{k}\widehat{C}\left(
z_{k}\right) \left( v_{k}\right) +O\left( s^{3}t_{k}^{\frac{3}{2}}\left\vert
v_{k}\right\vert ^{3}\right) \right]  \\
&=&1+s^{2}\left\vert v_{k}\right\vert ^{2}\left( -2+t_{k}\widehat{C}\left(
z_{k}\right) \left( \frac{v_{k}}{\left\vert v_{k}\right\vert }\right)
\right) +O\left( s^{3}t_{k}^{\frac{3}{2}}\left\vert v_{k}\right\vert
^{3}\right) ,
\end{eqnarray*}%
where the function $\widehat{C}\left( z\right) $ is constructed from $%
\widetilde{A}$, $\widetilde{B}$, and $\widetilde{E}$, and $\widehat{C}\left(
z\right) \left( v\right) $ is quadratic in $v$. Hence as $k\rightarrow
\infty $, 
\begin{equation*}
0=f_{k}^{^{\prime \prime }}\left( s_{k}\right) =2\left\vert v_{k}\right\vert
^{2}\left[ -2+O\left( t_{k}\right) \right] +O\left( s_{k}t_{k}^{\frac{3}{2}%
}\left\vert v_{k}\right\vert ^{3}\right) =2\left\vert v_{k}\right\vert ^{2}%
\left[ -2+O\left( 1\right) \right] .
\end{equation*}%
But by our assumption $v_{k}\neq 0$, so the right hand side will be nonzero
for large $k$, a contradiction.

The proof of the one-to-one property of $I_{t}:M\rightarrow \mathbb{R}%
^{q\left( t\right) }$ is almost identical to that of $\Psi _{t}$. This is
because we have $\left\Vert I_{t}-\Psi _{t}\right\Vert _{C^{k,\alpha }\left(
M\right) }\leq Ct^{l+\frac{1}{2}-\frac{k+\alpha }{2}}$ for $k\geq 2$ from
Proposition \ref{isom-immersion}, so the function%
\begin{equation*}
f_{k}\left( s\right) :=\frac{\left\vert \left\langle I_{t}\left(
x_{k}\right) ,I_{t}\left( y_{k}^{s}\right) \right\rangle \right\vert ^{2}}{%
\left\vert I_{t}\left( x_{k}\right) \right\vert ^{2}\left\vert I_{t}\left(
y_{k}^{s}\right) \right\vert ^{2}}
\end{equation*}%
has the same properties as the above functions $f_{k}\left( s\right) $ for $%
\Psi _{t}$, up to the second order derivatives. This is because if we
replace $H\left( t,x,y\right) $ for $\Psi _{t}$ by $\left\langle I_{t}\left(
x\right) ,I_{t}\left( y\right) \right\rangle $ for $I_{t}$ in the key
estimates $\left( \ref{qt-agree}\right) $, $\left( \ref{heat-agree}\right) $%
, $\left( \ref{HxxU}\right) $, $\left( \ref{HyyU}\right) $ and $\left( \ref%
{HxyU}\right) $, the argument still holds. The proposition follows.
\end{proof}

\begin{corollary}
\label{finite-eigen-distinguish}Let $\left( M,g\right) $ be a compact
Riemannian manifold with smooth metric $g$. Then there exists an integer $%
N_{0}>0$ depending on $g$, such that the first $N_{0}$ eigenfunctions $%
\left\{ \phi _{j}\right\} _{j=1}^{N_{0}}$ can distinguish any two points on $%
M\,$, i.e. for any $x\neq y$ on $M$, there exists some $j_{0}\in \left\{
1,2,\cdots N_{0}\right\} $, such that $\phi _{j_{0}}\left( x\right) \neq
\phi _{j_{0}}\left( y\right) $.
\end{corollary}

\begin{proof}
Take $N_{0}=q\left( \delta _{0}\right) $ in the above Proposition, and note
that $\Psi _{t}^{N_{0}}\left( x\right) =\Psi _{t}^{N_{0}}\left( y\right)
\Longleftrightarrow \phi _{j}\left( x\right) =\phi _{j}\left( y\right) $ for 
$1\leq j\leq N_{0}$.
\end{proof}

\section{Geometry of the embedded images in $\ell ^{2}$ (and $\mathbb{R}%
^{q\left( t\right) }$)\label{asymp-geom}}

In this section we study the geometry of the embedded images $\Psi
_{t}\left( M\right) $ and $I_{t}\left( M\right) $ in $\ell ^{2}$. We first
combine Theorem \ref{Uni-indep} and Proposition \ref{heat-asymp} to give the
following consequence on the second fundamental form and mean curvature of
the embedded image $\Psi _{t}\left( M\right) \subset \ell ^{2}$:

\begin{corollary}
\label{second-fund-form-mean-curvature}For any $x\in M$, let $\left(
x_{1},\cdots ,x_{n}\right) $ be the normal coordinates near $x$. The second
fundamental form $A\left( x,t\right) =\Sigma _{1\leq i\leq j\leq
n}h_{ij}\left( x,t\right) dx^{i}dx^{j}$ of the submanifold $\Psi _{t}\left(
M\right) \subset $ $\ell ^{2}$ can be written as%
\begin{equation*}
A\left( x,t\right) =\frac{1}{\sqrt{2t}}\left( \Sigma _{i=1}^{n}\sqrt{3}%
a_{ii}\left( x,t\right) \left( dx^{i}\right) ^{2}+\Sigma _{1\leq j<k\leq
n}2a_{jk}\left( x,t\right) dx^{j}dx^{k}\right) ,
\end{equation*}%
where $a_{jk}\left( x,t\right) $ $\left( 1\leq j\leq k\leq n\right) $ are
vectors in $\ell ^{2}$. Then as $t\rightarrow 0_{+}$,

\begin{enumerate}
\item For any two subsets $\left\{ i,j\right\} $ and $\left\{ k,l\right\} $ $%
\subset \left\{ 1,2,\cdots ,n\right\} $, $\ $%
\begin{eqnarray}
\left\langle a_{ij},a_{ij}\right\rangle &\rightarrow &1,  \label{aijij} \\
\left\langle a_{ij},a_{kl}\right\rangle &\rightarrow &0,\text{ if }\left\{
i,j\right\} \neq \left\{ k,l\right\} \text{ and }\left\{ i,k\right\} \neq
\left\{ j,l\right\} ,  \label{aijkl} \\
\left\langle a_{ii},a_{jj}\right\rangle &\rightarrow &\frac{1}{3},\text{ if }%
i\neq j.  \label{aiijj}
\end{eqnarray}

\item The mean curvature vector $H\left( x,t\right) =\frac{1}{n}\Sigma
_{i=1}^{n}h_{ii}\left( x,t\right) $, after scaled by a factor $\sqrt{t}$,
converges to constant length: 
\begin{equation}
\sqrt{t}\left\vert H\left( x,t\right) \right\vert \rightarrow \sqrt{\frac{n+2%
}{2n}}.  \label{mean-curvature-length}
\end{equation}%
The convergence is uniform for all $x$ on $M$ in the $C^{r}$-norm for any $%
r\geq 0$.
\end{enumerate}
\end{corollary}

\begin{proof}
From Proposition \ref{heat-asymp} we have%
\begin{equation}
\left\vert \left\langle \nabla _{i}\nabla _{j}\Phi _{t},\nabla _{k}\Phi
_{t}\right\rangle \left( x\right) \right\vert =O\left( t^{-n/2}\right) .
\label{1st-2nd-deri-orthogonal}
\end{equation}

Therefore for the \emph{normalized heat kernel embedding }$\Psi _{t}=\sqrt{2}%
\left( 4\pi \right) ^{n/4}t^{\frac{n+2}{4}}\cdot \Phi _{t}$, its first
derivative and second derivative vectors become orthogonal as $t\rightarrow
0_{+}$ by $\left( \ref{1st-2nd-deri-orthogonal}\right) $: 
\begin{equation*}
\left\vert \left\langle \nabla _{i}\nabla _{j}\Psi _{t},\nabla _{k}\Psi
_{t}\right\rangle \left( x\right) \right\vert \rightarrow 2\left( 4\pi
\right) ^{n/2}t^{\frac{n+2}{2}}\cdot t^{-n/2}\cdot O\left( 1\right) =Ct\cdot
O\left( 1\right) \rightarrow 0.
\end{equation*}%
So as $t\rightarrow 0_{+}$, the \emph{second fundamental form} at $\Psi
_{t}\left( x\right) $\ on $\Psi _{t}\left( M\right) \subset \ell ^{2}$ is
approximated by the second order terms in the Taylor expansion of $\Psi
_{t}:M\rightarrow \ell ^{2}$ near $x$ on $M$, i.e. 
\begin{equation*}
\lim_{t\rightarrow 0_{+}}\left[ A\left( x,t\right) -\left( \Sigma _{1\leq
i\leq j\leq n}\nabla _{i}\nabla _{j}\Psi _{t}\left( x,t\right)
dx^{i}dx^{j}\right) \right] =0.
\end{equation*}

From Proposition \ref{heat-asymp}, we have%
\begin{eqnarray}
\left\langle \nabla _{j}\nabla _{i}\Psi _{t},\nabla _{m}\nabla _{k}\Psi
_{t}\right\rangle \left( x\right) &\rightarrow &2\left( 4\pi \right)
^{n/2}t^{\frac{n+2}{2}}\cdot \left( \frac{1}{2t}\right) ^{2}\frac{1}{\left(
4\pi t\right) ^{n/2}}\cdot \left( \delta _{ij}\delta _{km}+\delta
_{im}\delta _{jk}+\delta _{ik}\delta _{jm}\right)  \notag \\
&=&\frac{1}{2t}\left( \delta _{ij}\delta _{km}+\delta _{im}\delta
_{jk}+\delta _{ik}\delta _{jm}\right) .  \label{psi_ijkm}
\end{eqnarray}%
Especially as $t\rightarrow 0_{+}$, 
\begin{equation*}
\left\vert \frac{1}{\sqrt{2t}}\cdot \sqrt{3}a_{ii}\left( x,t\right)
\right\vert =\left\vert \nabla _{i}\nabla _{i}\Psi _{t}\right\vert
\rightarrow \frac{\sqrt{3}}{\sqrt{2t}}\text{, and }\left\vert \frac{1}{\sqrt{%
2t}}\cdot a_{jk}\left( x,t\right) \right\vert =\left\vert \nabla _{j}\nabla
_{k}\Psi _{t}\right\vert \rightarrow \frac{1}{\sqrt{2t}}\text{ (}j\neq k%
\text{),}
\end{equation*}%
so $\left( \ref{aijij}\right) $ follows. Similarly, $\left( \ref{aijkl}%
\right) $ and $\left( \ref{aiijj}\right) $ follow from $\left( \ref{psi_ijkm}%
\right) $. For the mean curvature, we have 
\begin{equation*}
H\left( x,t\right) =\frac{1}{n}\Sigma _{i=1}^{n}h_{ii}\left( x,t\right) =%
\frac{1}{n}\frac{\sqrt{3}}{\sqrt{2t}}\left( \Sigma _{i=1}^{n}a_{ii}\left(
x,t\right) \right) .
\end{equation*}%
Using $\left\vert a_{ii}\right\vert \rightarrow 1$ and $\left\langle
a_{ii},a_{jj}\right\rangle \rightarrow \frac{1}{3}$ as $t\rightarrow 0_{+}$,
we have%
\begin{equation*}
\left\vert H\left( x,t\right) \right\vert ^{2}\rightarrow \frac{1}{n^{2}}%
\frac{3}{2t}\left( n\cdot 1+n\left( n-1\right) \cdot \frac{1}{3}\right) =%
\frac{1}{2t}\frac{n+2}{n},
\end{equation*}%
so $\left( \ref{mean-curvature-length}\right) $ follows.
\end{proof}

\begin{remark}
\label{Comments-2nd-funda-form}In the above corollary,

\begin{enumerate}
\item It is unknown if $\lim_{t\rightarrow 0_{+}}a_{jk}\left( x,t\right) $
exists, but $\lim_{t\rightarrow 0_{+}}\left\langle a_{ij}\left( x,t\right)
,a_{kl}\left( x,t\right) \right\rangle $ exists. There exists isometry $%
I\left( x,t\right) :\ell ^{2}\rightarrow \ell ^{2}$, such that 
\begin{equation*}
a_{jk}:=\lim_{t\rightarrow 0_{+}}I\left( x,t\right) \cdot a_{jk}\left(
x,t\right) \text{ }\left( 1\leq j\leq k\leq n\right)
\end{equation*}%
exists, and $\left\{ a_{jk}\right\} _{1\leq j\leq k\leq n}$ is a fixed basis
in $\mathbb{R}^{\frac{n\left( n+1\right) }{2}}\subset \ell ^{2}$ satisfying
the inner product relations in item 1 of the above corollary;

\item The length of the mean curvature of $\Psi _{t}\left( M\right) \subset $
$\ell ^{2}$ converges to constant on $M$, but the constant is large (of
order $t^{-\frac{1}{2}}$) as $t\rightarrow 0_{+}$ by $\left( \ref%
{mean-curvature-length}\right) $. Intuitively, this is because the embedding 
$\Psi _{t}$ uses more and more high frequency eigenfunctions in its $\ell
^{2}$ norm as $t\rightarrow 0_{+}$, making the image $\Psi _{t}\left(
M\right) $ evenly oscillating at all $x$ on $M$.

\item By Proposition \ref{isom-truncation-high-jet}, the above results still
hold if we replace $\ell ^{2}$ by $\mathbb{R}^{q\left( t\right) }$, for $%
q\left( t\right) \geq Ct^{-\frac{n}{2}-\rho }$ and sufficiently small $t>0$.
\end{enumerate}
\end{remark}

As for our isometric embeddings $I_{t}:M\rightarrow \mathbb{R}^{q\left(
t\right) }$ ($q\left( t\right) \geq Ct^{-\frac{n}{2}-\rho }$ or $q\left(
t\right) =\infty $), they are obtained by $C^{k,\alpha }$-perturbation of $%
\tilde{\Psi}_{t}:$ $M\rightarrow \mathbb{R}^{q\left( t\right) }$ of order $%
O\left( t^{l+\frac{1}{2}-\frac{k+\alpha }{2}}\right) $, with $k\geq 2$.
Since the second fundamental form and mean curvature of any embedding $%
f:M\rightarrow \mathbb{R}^{q}$ are determined by up to the \emph{second order%
} derivatives of $f$, the statements in Corollary \ref%
{second-fund-form-mean-curvature} also hold for the isometric embedding
image $I_{t}\left( M\right) \subset \mathbb{R}^{q\left( t\right) }$,
noticing that $\tilde{\Psi}_{t}:=\Psi _{t,g\left( t\right) }$ is the heat
kernel embedding map for the metric $g\left( t\right) $ on $M$. For the same
reason, for any $0\leq r\leq k$, the $r$-jet relations of $I_{t}$ as $%
t\rightarrow 0_{+}$ are the same as those for $\Psi _{t}$ in \cite{Z}. This
gives many constraints of the image $I_{t}\left( M\right) \subset \mathbb{R}%
^{q\left( t\right) }$.

\section{Example $M=S^{1\text{ }}$\label{examples}}

We can write down the matrix $P\left( u\right) $ and $E\left( u\right) $
explicitly in the $n=1$ case, with $M=S^{1}\simeq \left[ 0,2\pi \right]
/\sim $. This example seems trivial but has almost all features of general
cases, so we use it to illustrate the proofs of our main results. 

For the eigenvalue $\lambda _{2k-1}=\lambda _{2k}=k^{2}$, the $L^{2}$
orthonormal eigenfunctions are pairs $\phi _{2k-1}\left( x\right) =\frac{1}{%
\sqrt{\pi }}\cos kx$ and $\phi _{2k}\left( x\right) =\frac{1}{\sqrt{\pi }}%
\sin kx$. The heat kernel embedding $u:S^{1}\rightarrow $ $\mathbb{R}^{2q}$
is 
\begin{equation*}
u\left( x\right) =\sqrt{\frac{2}{\pi }}\left( 4\pi \right) ^{n/4}t^{\frac{n+2%
}{4}}\left\{ \left( e^{-\frac{k^{2}}{2}t}\cos kx,e^{-\frac{k^{2}}{2}t}\sin
kx\right) \right\} _{1\leq k\leq q},
\end{equation*}
so the system $\left( \ref{linear-system}\right) $ becomes%
\begin{equation*}
R_{1}\left( x\right) \cdot v=h_{1}\left( x\right) \text{, }R_{2}\left(
x\right) \cdot v=f_{11}\left( x\right) ,
\end{equation*}%
with the two row vectors $R_{1}$ and $R_{2}$ being 
\begin{eqnarray*}
R_{1}\left( x\right)  &=&\sqrt{\frac{2}{\pi }}\left( 4\pi \right) ^{n/4}t^{%
\frac{n+2}{4}}\left\{ \left( -e^{-\frac{k^{2}}{2}t}k\sin kx,e^{-\frac{k^{2}}{%
2}t}k\cos kx\right) \right\} _{1\leq k\leq q}, \\
R_{2}\left( x\right)  &=&\sqrt{\frac{2}{\pi }}\left( 4\pi \right) ^{n/4}t^{%
\frac{n+2}{4}}\left\{ \left( -e^{-\frac{k^{2}}{2}t}k^{2}\cos kx,-e^{-\frac{%
k^{2}}{2}t}k^{2}\sin kx\right) \right\} _{1\leq k\leq q}.
\end{eqnarray*}%
It is easy to check $R_{1}\left( x\right) $ and $R_{2}\left( x\right) $ are 
\emph{orthogonal}, so the solution $v$ with minimal Euclidean norm is $%
v\left( x\right) =\frac{h_{1}\left( x\right) }{\left\vert R_{1}\left(
x\right) \right\vert ^{2}}R_{1}\left( x\right) +\frac{f_{11}\left( x\right) 
}{\left\vert R_{2}\left( x\right) \right\vert ^{2}}R_{2}\left( x\right) $.
Hence 
\begin{equation*}
E\left( u\right) =\left[ 
\begin{array}{c}
R_{1}\left( x\right) /\left\vert R_{1}\left( x\right) \right\vert ^{2} \\ 
R_{2}\left( x\right) /\left\vert R_{2}\left( x\right) \right\vert ^{2}%
\end{array}%
\right] \text{, }\ \ \ \text{ }E\left( u\right) \left( 0,f_{11}\right) =%
\frac{f_{11}\left( x\right) }{\left\vert R_{2}\left( x\right) \right\vert
^{2}}R_{2}\left( x\right) .
\end{equation*}%
Recall the following well-known

\begin{lemma}
\label{Rie-sum}$\lim_{t\rightarrow 0_{+}}t^{\frac{m+1}{2}}\Sigma
_{k=1}^{\infty }k^{m}e^{-k^{2}t}=\int_{0}^{\infty }\mu ^{m}e^{-\mu ^{2}}d\mu 
$, and $\Sigma _{k=1}^{\infty }k^{m}e^{-k^{2}t}\leq Kt^{-\frac{m+1}{2}}$,
where the constant $K=\int_{0}^{\infty }\mu ^{m}e^{-\mu ^{2}}d\mu $.
\end{lemma}

For $q$ large, using $\int_{0}^{\infty }\mu ^{2}e^{-\mu ^{2}}d\mu =\frac{%
\sqrt{\pi }}{4}$ and $\int_{0}^{\infty }\mu ^{4}e^{-\mu ^{2}}d\mu =\frac{3%
\sqrt{\pi }}{8}$, we have 
\begin{eqnarray*}
\left\vert R_{1}\left( x\right) \right\vert ^{2} &=&\frac{2}{\pi }\left(
4\pi \right) ^{n/2}t^{\frac{n+2}{2}}\Sigma
_{k=1}^{q}k^{2}e^{-k^{2}t}\rightarrow \frac{2}{\pi }\left( 4\pi \right)
^{1/2}t^{\frac{3}{2}}\cdot t^{-\frac{3}{2}}\int_{0}^{\infty }\mu ^{2}e^{-\mu
^{2}}d\mu =1, \\
\left\vert R_{2}\left( x\right) \right\vert ^{2} &=&\frac{2}{\pi }\left(
4\pi \right) ^{n/2}t^{\frac{n+2}{2}}\Sigma
_{k=1}^{q}k^{4}e^{-k^{2}t}\rightarrow \frac{2}{\pi }\left( 4\pi \right)
^{1/2}t^{\frac{3}{2}}\cdot t^{-\frac{5}{2}}\int_{0}^{\infty }\mu ^{4}e^{-\mu
^{2}}d\mu =\frac{3}{2}t^{-1}.
\end{eqnarray*}%
These agree with $\Psi _{t}^{\ast }g_{can}\rightarrow g$ in $\left( \ref%
{asymp-isom}\right) $ and the mean curvature length $\left\vert H\left(
x,t\right) \right\vert \rightarrow \sqrt{\frac{1+2}{2\cdot 1}}t^{-\frac{1}{2}%
\text{ }}$in Corollary \ref{second-funda-form} respectively. Thus for $%
q=q\left( t\right) $ large, in $C^{3}$ convergence we have 
\begin{equation}
E\left( u\right) \rightarrow \left[ 
\begin{array}{c}
R_{1}\left( x\right)  \\ 
\frac{2}{3}tR_{2}\left( x\right) 
\end{array}%
\right] \text{, \ \ \ }E\left( u\right) \left( 0,f_{11}\right) \rightarrow 
\frac{2t}{3}R_{2}\left( x\right) f_{11}\left( x\right) .  \label{Eu_f11}
\end{equation}%
We have the $C^{2}$ norm (according to Definition \ref%
{Schauder-norm-vec-valued-fcn}) for vector-valued functions 
\begin{eqnarray*}
\left\Vert R_{1}\left( x\right) \right\Vert _{C^{2}\left( M\right) } &=&%
\sqrt{\frac{2}{\pi }}\left( 4\pi \right) ^{n/4}t^{\frac{n+2}{4}}\left[
\Sigma _{k=1}^{q}e^{-k^{2}t}k^{6}\right] ^{\frac{1}{2}}\rightarrow Ct^{\frac{%
3}{4}}t^{-\frac{7}{4}}=Ct^{-1} \\
\left\Vert R_{2}\left( x\right) \right\Vert _{C^{2}\left( M\right) } &=&%
\sqrt{\frac{2}{\pi }}\left( 4\pi \right) ^{n/4}t^{\frac{n+2}{4}}\left[
\Sigma _{k=1}^{q}e^{-k^{2}t}k^{8}\right] ^{\frac{1}{2}}\rightarrow Ct^{\frac{%
3}{4}}t^{-\frac{9}{4}}=Ct^{-\frac{3}{2}} \\
\left\Vert E\left( u\right) \right\Vert _{C^{2}\left( M\right) } &=&\sqrt{%
\frac{2}{\pi }}\left( 4\pi \right) ^{n/4}t^{\frac{n+2}{4}}\cdot \left[
\Sigma _{k=1}^{q}e^{-k^{2}t}k^{6}+t^{2}\cdot \Sigma
_{k=1}^{q}e^{-k^{2}t}k^{8}\right] ^{1/2}\rightarrow Ct^{-1}.
\end{eqnarray*}

Notice that for $S^{1}$, the curvature tensor $R\equiv 0$, so $Ric_{g}-\frac{%
1}{2}g\cdot S_{g}\equiv 0$. By \cite{BBG} Theorem 5 we have $f_{11}=\Psi
_{t}^{\ast }g_{can}-g=O\left( t^{2}\right) $ as $t\rightarrow 0_{+}$ (it may
be of higher vanishing order $O\left( t^{p}\right) $ for some $p>2\,$, but
here we only use $O\left( t^{2}\right) $ to illustrate our method). So by $%
\left( \ref{Eu_f11}\right) $ we have%
\begin{equation*}
\left\Vert E\left( u\right) \left( 0,f_{11}\right) \right\Vert _{C^{2}\left(
M\right) }\rightarrow \left\Vert \frac{2t}{3}\cdot t^{2}R_{2}\left( x\right)
\right\Vert _{C^{2}\left( M\right) }=Ct^{3}\cdot \left\Vert R_{2}\left(
x\right) \right\Vert _{C^{2}\left( M\right) }=Ct^{\frac{3}{2}}.
\end{equation*}%
Then 
\begin{equation*}
\left\Vert E\left( u\right) \right\Vert _{C^{2}\left( M\right) }\left\Vert
E\left( u\right) \left( 0,f_{11}\right) \right\Vert _{C^{2}\left( M\right)
}\rightarrow Ct^{-1}\cdot Ct^{\frac{3}{2}}=Ct^{\frac{1}{2}},
\end{equation*}%
and similarly (using the interpolation technique in Lemma \ref%
{eigen-deri-est} to estimate the $C^{2,\alpha }$-norm from the $C^{2}$ and $%
C^{3}$-norms)%
\begin{equation*}
\left\Vert E\left( u\right) \right\Vert _{C^{2,\alpha }\left( M\right)
}\left\Vert E\left( u\right) \left( 0,f_{11}\right) \right\Vert
_{C^{2,\alpha }\left( M\right) }\rightarrow Ct^{-1-\frac{\alpha }{2}}\cdot
Ct^{\frac{3}{2}-\frac{\alpha }{2}}=Ct^{\frac{1}{2}-\alpha }\rightarrow 0
\end{equation*}%
for $0<\alpha <\frac{1}{2}$. We see the estimates of the orders are exactly
the same as obtained by the off-diagonal expansion of heat kernel method.

By G\"{u}nther's implicit function theorem we obtain the isometric
embeddings of $S^{1}$ into $\mathbb{R}^{q\left( t\right) }$.

\section{Appendix\label{IFT2}}

\subsection{The implicit function theorem\label{IFT_general}}

If we apply the following Proposition A.3.4. of \cite{MS} (an abstract
implicit function theorem) to the nonlinear function $F:C^{k,\alpha }\left(
M,\mathbb{R}^{q}\right) \rightarrow C^{k,\alpha }\left( M,\mathbb{R}%
^{q}\right) $ in Section \ref{IFT}, we recover G\"{u}nther's implicit
function theorem (Theorem \ref{Gunther-Implicit-function-thm}), and obtain
the following more information: First, the constant $\theta $ in his Theorem
can be made explicit in $\left( \ref{Gunther-constant}\right) $; Second, the
needed perturbation of $\Psi _{t}$ can be shown of order $O\left( t^{l+\frac{%
1}{2}-\frac{k+\alpha }{2}}\right) $ in the $C^{k,\alpha }$ norm.

\begin{proposition}
(Proposition A.3.4. of \cite{MS}) Let $X,Y$ be Banach spaces and $U$ be an
open set in $X$. The map $F:X\rightarrow Y$ is continuous differentiable.
For $x_{0}\in U,$ $D:=dF(x_{0}):X\rightarrow Y$ is surjective and has a
bounded linear right inverse $Q:Y\rightarrow X$, with $\left\Vert
Q\right\Vert \leq c$. Suppose that there exists $\delta >0$ such that $x\in
B_{\delta }(x_{0})\subset U$%
\begin{equation*}
x\in B_{\delta }(x_{0})\subset U\implies \left\Vert dF(x)-D\right\Vert \leq 
\frac{1}{2c}.
\end{equation*}%
Suppose $\left\Vert F(x_{0})\right\Vert \leq \frac{\delta }{4c}$, then there
exists a unique $x\in B_{\delta }(x_{0})$ such that%
\begin{equation*}
F(x)=0,~x-x_{0}\in \text{Image}Q,~\left\Vert x-x_{0}\right\Vert \leq
2c\left\Vert F(x_{0})\right\Vert .
\end{equation*}
\end{proposition}

Applying the above proposition to our case, we have (following their
notations)%
\begin{eqnarray*}
X &=&C^{k,\alpha }\left( M,\mathbb{R}^{q}\right) \text{, }Y=C^{k,\alpha
}\left( M,\mathbb{R}^{q}\right) \text{, }F:X\rightarrow Y, \\
F\left( v\right)  &=&v-E\left( \tilde{\Psi}_{t}\right) \left( 0,f\right)
+E\left( \tilde{\Psi}_{t}\right) \left( \left[ Q_{i}\left( \tilde{\Psi}%
_{t}\right) \left( v,v\right) \right] ,\left[ Q_{jk}\left( \tilde{\Psi}%
_{t}\right) \left( v,v\right) \right] \right) , \\
x_{0} &=&0\text{, }x\text{ solution, }F\left( x\right) =0,
\end{eqnarray*}%
\begin{eqnarray*}
\text{ }c &=&\left\Vert \left( dF\left( 0\right) \right) ^{-1}\right\Vert
=\left\Vert \left( id\right) ^{-1}\right\Vert =1, \\
\left\Vert F\left( 0\right) \right\Vert  &=&\left\Vert E\left( \tilde{\Psi}%
_{t}\right) \left( 0,f\right) \right\Vert _{C^{k,\alpha }\left( M,\mathbb{R}%
^{q}\right) }\leq C_{E}^{2}Gt^{l+\frac{1}{2}-\frac{k+\alpha }{2}}\text{ (by }%
\left( \text{\ref{E-error}}\right) \text{),} \\
\left\Vert dF\left( v\right) -dF\left( 0\right) \right\Vert  &\leq
&\left\Vert E\left( \tilde{\Psi}_{t}\right) \right\Vert _{C^{k,\alpha
}\left( M\right) }\cdot \Gamma \left( \Lambda _{0},k,\alpha ,\left\Vert
R\right\Vert _{C^{1}}\right) \cdot \left\Vert v\right\Vert _{C^{k,\alpha
}\left( M,\mathbb{R}^{q}\right) },
\end{eqnarray*}%
where the last inequality is from $\left( \ref{quadratic-Q-bd}\right) $.
Since the $\delta $ should satisfy%
\begin{equation*}
\left\Vert v-0\right\Vert \leq \delta \Longrightarrow \left\Vert dF\left(
v\right) -dF\left( 0\right) \right\Vert \leq \frac{1}{2c}\text{,}
\end{equation*}%
we can take 
\begin{equation*}
\text{ }\delta =\frac{1}{2\left\Vert E\left( \tilde{\Psi}_{t}\right)
\right\Vert _{C^{k,\alpha }\left( M\right) }\cdot \Gamma \left( \Lambda
_{0},k,\alpha ,\left\Vert R\right\Vert _{C^{1}}\right) }\geq \frac{t^{\frac{%
k+\alpha }{2}}}{2C_{E}\cdot \Gamma \left( \Lambda _{0},k,\alpha ,\left\Vert
R\right\Vert _{C^{1}}\right) }\text{(by }\left( \text{\ref{C_E}}\right) 
\text{)}.
\end{equation*}%
The condition 
\begin{equation}
\left\Vert F\left( 0\right) \right\Vert \leq \frac{\delta }{4c}
\label{IFT-cond}
\end{equation}%
is translated to%
\begin{equation*}
\left\Vert E\left( \tilde{\Psi}_{t}\right) \left( 0,f\right) \right\Vert
_{C^{k,\alpha }\left( M,\mathbb{R}^{q}\right) }\leq \frac{1}{4}\cdot \frac{1%
}{2\left\Vert E\left( \tilde{\Psi}_{t}\right) \right\Vert _{C^{k,\alpha
}\left( M\right) }\cdot \Gamma \left( \Lambda _{0},k,\alpha ,\left\Vert
R\right\Vert _{C^{1}}\right) },
\end{equation*}%
or%
\begin{equation}
\left\Vert E\left( \tilde{\Psi}_{t}\right) \left( 0,f\right) \right\Vert
_{C^{k,\alpha }\left( M,\mathbb{R}^{q}\right) }\left\Vert E\left( \tilde{\Psi%
}_{t}\right) \right\Vert _{C^{k,\alpha }\left( M\right) }\leq \left( 8\Gamma
\left( \Lambda _{0},k,\alpha ,\left\Vert R\right\Vert _{C^{1}}\right)
\right) ^{-1}:=\theta .  \label{Gunther-constant}
\end{equation}%
So the constant $\theta $ is essentially determined by $\Gamma \left(
\Lambda _{0},k,\alpha ,\left\Vert R\right\Vert _{C^{1}}\right) $ in $\left( %
\ref{Gamma}\right) $,$\ $which in turn depends on $\left\Vert R\right\Vert
_{C^{1}}$ and $\sigma \left( \Lambda _{0},\alpha ,M\right) =\left\Vert
\left( \Delta _{\left( r\right) }-\Lambda _{0}\right) ^{-1}\right\Vert _{%
\mathrm{op}}$ of the smoothing operators (Note $C\left( k,\alpha ,M\right)
=n^{k}$ in Lemma \ref{product-inequ} is independent on $g$). In terms of $t$
the condition $\left( \ref{IFT-cond}\right) $ is%
\begin{equation*}
C_{E}^{2}Gt^{l+\frac{1}{2}-\frac{k+\alpha }{2}}\leq \frac{t^{\frac{k+\alpha 
}{2}}}{8C_{E}\Gamma \left( \Lambda _{0},k,\alpha ,\left\Vert R\right\Vert
_{C^{1}}\right) },
\end{equation*}%
i.e. 
\begin{equation}
t\leq \left( 8C_{E}^{3}\cdot \Gamma \left( \Lambda _{0},k,\alpha ,\left\Vert
R\right\Vert _{C^{1}}\right) \cdot G\right) ^{-\frac{1}{\left( l+1/2\right)
-\left( k+\alpha \right) }}:=t_{0}.  \label{t-condition}
\end{equation}%
So we see the smallness of $t$ in our implicit function theorem depends on 

\begin{enumerate}
\item the constant $C_{E}$ in $\left( \ref{C_E}\right) $ in the $C^{k,\alpha
}$-norm estimate for $E\left( \tilde{\Psi}_{t}\right) $ (essentially the
derivative estimates of $\tilde{\Psi}_{t}$);

\item the constant $\Gamma \left( \Lambda _{0},k,\alpha ,\left\Vert
R\right\Vert _{C^{1}}\right) $ in $\left( \ref{Gamma}\right) $ related to
the smoothing operators $\left( \Delta _{\left( r\right) }-\Lambda
_{0}\right) ^{-1}$;

\item the constant $G$ in $\left( \ref{G-constant}\right) $ from the near
diagonal expansion of the heat kernel of $\left( M,g\right) $. 
\end{enumerate}

These constants are related to the dimension, diameter, volume and curvature
bounds of $\left( M,g\right) $. Since the exponent $-\frac{1}{\left(
l+1/2\right) -\left( k+\alpha \right) }$ in $\left( \ref{t-condition}\right) 
$ is negative, we see the smaller the constants $C_{E}$, $\Gamma \left(
\Lambda _{0},k,\alpha ,\left\Vert R\right\Vert _{C^{1}}\right) $ and $G$
are, the smaller the embedding dimension $q\left( t_{0}\right) $ is.

If we know $t_{0}$, we can obtain the estimate of the minimal embedding
dimension $q\left( t_{0}\right) \geq t_{0}^{-\frac{n}{2}-\rho }$. From the
above Proposition the solution $x$ satisfies $\left\Vert x\right\Vert \leq
2c\left\Vert F\left( 0\right) \right\Vert $, i.e. the perturbation of $%
\tilde{\Psi}_{t}$ is of order%
\begin{equation*}
\left\Vert x\right\Vert _{C^{k,\alpha }\left( M,\mathbb{R}^{q}\right) }\leq
2C_{E}^{2}Gt^{l+\frac{1}{2}-\frac{k+\alpha }{2}}=O\left( t^{l+\frac{1}{2}-%
\frac{k+\alpha }{2}}\right) \text{.}
\end{equation*}

\subsection{\protect\bigskip The quadratic remainder\label{remainder}}

To estimate the constant $G$ in $\left( \ref{G-constant}\right) $, we refine
the asymptotic formula $\left( \ref{asymp-isom}\right) $ in \cite{BBG} to
the quadratic terms. The following lemma is well-known in physics
literature. 

\begin{lemma}
\label{quadratic-asymp} Same notations as in Proposition \ref{heat-asymp}.
Then as $t\rightarrow 0_{+}$, we have%
\begin{eqnarray}
&&\left( \Psi _{t}^{\ast }g_{can}\right) \left( x\right) =g\left( x\right) +%
\frac{t}{3}\left( \frac{1}{2}S_{g}\cdot g-Ric_{g}\right)   \notag \\
&&+t^{2}\left[ u_{2}\left( x,x\right) +\Sigma _{i,j=1}^{n}2\partial _{\bar{j}%
}\partial _{i}u_{1}\left( x,y\right) |_{x=y}\right] dx^{i}dx^{j}+O\left(
t^{3}\right) .  \label{refined-isom-asymp}
\end{eqnarray}
\end{lemma}

\begin{proof}
From $\left( \ref{MP-expansion}\right) $ we have%
\begin{eqnarray}
\partial _{i}H\left( t,x,y\right)  &=&\frac{1}{\left( 4\pi t\right) ^{n/2}}%
e^{-\frac{r^{2}}{4t}}\left[ -\frac{\partial _{i}\left( r^{2}\right) }{4t}%
U+\partial _{i}U\right] ,  \label{1st-order} \\
\partial _{j}\partial _{i}H\left( t,x,y\right)  &=&\frac{1}{\left( 4\pi
t\right) ^{n/2}}e^{-\frac{r^{2}}{4t}}\left[ -\frac{\partial _{j}\left(
r^{2}\right) }{4t}\left( -\frac{\partial _{i}\left( r^{2}\right) }{4t}%
U+\partial _{i}U\right) \right.   \notag \\
&&\left. +\left( -\frac{\partial _{j}\partial _{i}\left( r^{2}\right) }{4t}U-%
\frac{\partial _{i}\left( r^{2}\right) }{4t}\partial _{j}U+\partial
_{j}\partial _{i}U\right) \right] .  \label{2nd-order}
\end{eqnarray}

From Lemma \ref{Lem:r} and $\left( \ref{U-expansion}\right) $, letting $x=y$
we have%
\begin{eqnarray*}
&&\left. \partial _{\bar{j}}\partial _{i}H\left( t,x,y\right) \right\vert
_{x=y} \\
&=&\frac{1}{\left( 4\pi t\right) ^{n/2}}\left[ -\frac{\partial _{\bar{j}%
}\partial _{i}\left( r^{2}\right) }{4t}\left(
u_{0}+tu_{1}+t^{2}u_{2}+O\left( t^{2}\right) \right) -\frac{\partial
_{i}\left( r^{2}\right) }{4t}\partial _{_{\bar{j}}}U+\partial _{_{\bar{j}%
}}\partial _{i}\left( u_{0}+tu_{1}+O\left( t\right) \right) \right] .
\end{eqnarray*}%
Letting $i=j$, using Lemma \ref{Lem:r} in the above identity, we have for $%
V_{i}=\frac{\partial }{\partial x^{i}}$,%
\begin{eqnarray*}
&&\left( \Psi _{t}^{\ast }g_{can}\right) \left( V_{i,}V_{i}\right) \left(
x\right)  \\
&=&2\left( 4\pi \right) ^{\frac{n}{2}}t^{\frac{n+2}{2}}\cdot \frac{1}{\left(
4\pi t\right) ^{n/2}}\left. \left[ -\frac{\partial _{\bar{\imath}}\partial
_{i}\left( r^{2}\right) }{4t}\left( u_{0}+tu_{1}+t^{2}u_{2}+O\left(
t^{2}\right) \right) +\partial _{\bar{\imath}}\partial _{i}\left(
u_{0}+tu_{1}+O\left( t\right) \right) \right] \right\vert _{x=y} \\
&=&\left[ g-\frac{t}{3}\left( Ric_{g}-\frac{1}{2}S_{g}\cdot g\right) \right]
\left( V_{i,}V_{i}\right) +t^{2}\left[ u_{2}\left( x,x\right) +2\partial _{%
\bar{\imath}}\partial _{i}u_{1}\left( x,y\right) |_{x=y}\right] +O\left(
t^{3}\right) .
\end{eqnarray*}%
Since $\left( \Psi _{t}^{\ast }g_{can}\right) \left( V,W\right) $ is
bilinear in $V$ and $W$, the proposition follows.
\end{proof}

Using the higher order expansion of $H\left( t,x,y\right) $ in terms of
curvature terms, it seems possible to make the quadratic terms in the above
lemma explicit. On p. 224\symbol{126}225 of \cite{BeGaM} and Theorem 3.3.1
of \cite{Gil} there is an explicit 
\begin{equation*}
u_{2}\left( x,x\right) =\frac{1}{180}\left\vert R_{g}\left( x\right)
\right\vert ^{2}-\frac{1}{180}\left\vert Ric_{g}\left( x\right) \right\vert
^{2}+\frac{1}{72}\left\vert S_{g}\left( x\right) \right\vert ^{2}-\frac{1}{30%
}\Delta _{g}S_{g}\left( x\right) ,
\end{equation*}%
where $R_{g}$ is the Riemannian curvature tensor, $\left\vert R_{g}\left(
x\right) \right\vert ^{2}=\Sigma _{1\leq i,j,k,l\leq n}\left\vert
R_{g}\left( x\right) \left( V_{i},V_{j},V_{k},V_{l}\right) \right\vert ^{2}$
for the basis $V_{i}=\frac{\partial }{\partial x^{i}}$ of normal coordinates 
$\left\{ x^{i}\right\} _{1\leq i\leq n}$ near $x$, similarly for $\left\vert
Ric_{g}\left( x\right) \right\vert ^{2}$. It remains to compute $\partial _{%
\bar{\imath}}\partial _{i}u_{1}\left( x,y\right) |_{x=y}$. There are physics
literatures  giving%
\begin{eqnarray}
&&\partial _{_{\bar{j}}}\partial _{i}u_{1}\left( x,y\right) |_{x=y}  \notag
\\
&=&\left[ \frac{1}{20}\partial _{j}\partial _{i}S_{g}-\frac{1}{60}\Delta
_{g}Ric_{g}\left( V_{j},V_{i}\right) +\frac{1}{36}S_{g}Ric_{g}\left(
V_{j,}V_{i}\right) \right. -\frac{1}{45}\Sigma _{k=1}^{n}Ric_{g}\left(
V_{j},V_{k}\right) Ric_{g}\left( V_{i},V_{k}\right)   \notag \\
&&\left. +\frac{1}{90}\Sigma _{k,l=1}^{n}\left( Ric_{g}\left(
V_{k},V_{l}\right) R\left( V_{k},V_{j},V_{l},V_{i}\right) +\left\vert
R\left( V_{k},V_{l},V_{j},V_{i}\right) \right\vert ^{2}\right) \right]
\left( x\right) .  \label{2nd-deri-u1}
\end{eqnarray}

Let's take $l=2$ in $\left( \ref{G-metric-expansion}\right) $. Then $%
h_{1}=-A_{1}\left( g\right) =\frac{1}{3}\left( Ric_{g}-\frac{1}{2}S_{g}\cdot
g\right) $, and $h_{j}=0$ for all $j\geq 2$. We have 
\begin{equation}
\left( \tilde{\Psi}_{t}^{\ast }g_{can}\right) =G\left( t,t\right) =g+t^{2}%
\left[ A_{1,1}\left( h_{1}\right) +A_{2}\left( g\right) \right] +O\left(
t^{3}\right) ,  \label{G-expansion-quadratic}
\end{equation}
where $A_{2}\left( g\right) =u_{2}\left( x,x\right) +\Sigma
_{i,j=1}^{n}2\partial _{\bar{j}}\partial _{i}u_{1}\left( x,y\right) |_{x=y}$
by the above Lemma, and $A_{1,1}\left( h_{1}\right) $ (first order variation
of $A_{1}\left( g\right) =\frac{1}{3}\left( \frac{1}{2}S_{g}\cdot
g-Ric_{g}\right) $) can be computed by the formulae of the variation of
curvature tensors (cf. \cite{Be}, Theorem 1.174): 
\begin{eqnarray*}
Ric_{g}^{\prime }h &=&\frac{1}{2}\Delta _{\left( 2\right) }h-\delta
_{g}^{\ast }\left( \delta _{g}h\right) -\frac{1}{2}\nabla _{g}d\left(
tr_{g}h\right) , \\
S_{g}^{\prime }h &=&\Delta _{g}\left( tr_{g}h\right) +\delta _{g}\left(
\delta _{g}h\right) -g\left( Ric_{g},h\right) ,
\end{eqnarray*}%
where $\delta _{g}$ is the divergence and $\delta _{g}^{\ast }$ is the
formal adjoint, and the notation \textquotedblleft $^{\prime }$%
\textquotedblright\ means the derivative with respect to the variation $h$
of $g$. 

If the derivation of $\left( \ref{2nd-deri-u1}\right) $ is rigorous, putting
all these into $\left( \ref{G-expansion-quadratic}\right) $, it appears that
we can control the constant $G$ in $\left( \ref{h-error}\right) $ by 
\begin{equation}
G\leq C\left( n\right) \left( \left\Vert Ric_{g}\right\Vert _{C^{4,\alpha
}\left( M\right) }+\left\Vert R_{g}\right\Vert _{C^{2,\alpha }\left(
M\right) }^{2}\right)   \label{G-estimate}
\end{equation}%
for small $t>0$, with a constant $C\left( n\right) $ only depending on $n$.

{\small Addresses:}

{\small Xiaowei Wang}

{\small Department of Mathematics \& Computer Science, Rutgers
University-Newark, Newark, NJ 07102}

{\small Email: xiaowwan@rutgers.edu }

\bigskip

{\small Ke Zhu}

{\small Department of Mathematics, Harvard University, Cambridge, MA 02138}

{\small Email: kzhu@math.harvard.edu}

\bigskip

\end{document}